\documentclass[twocolumn]{autart}    % Enable this line and disable the 

\usepackage{amsmath}
\usepackage{tikz}
\usetikzlibrary{arrows.meta,positioning,calc,decorations.pathmorphing}

\usepackage{graphicx} % Required for inserting images
\usepackage{enumitem}
\usepackage{comment}
 \usepackage{algorithm}
 \usepackage{algpseudocode}

\usepackage{csquotes}
\usepackage{xcolor}
\usepackage{tikz}
\usepackage{amsmath,amssymb}
\usetikzlibrary{arrows.meta, positioning, calc}

\newtheorem{theorem}{Theorem}
\newtheorem{proposition}{Proposition}
\newtheorem{assumption}{Assumption}
\newtheorem{lemma}{Lemma}
\theoremstyle{definition}
\newtheorem{definition}{Definition}
\newtheorem{corollary}{Corollary}
\theoremstyle{remark}

\definecolor{gridblue}{RGB}{44,95,140}
\definecolor{softblue}{RGB}{232,240,248}
\definecolor{boxblue}{RGB}{214,228,242}
\definecolor{accentteal}{RGB}{45,140,140}
\definecolor{textdark}{RGB}{35,35,35}

\definecolor{controlblue}{RGB}{35,90,150}
\definecolor{darkblue}{RGB}{20,60,110}
\definecolor{signalorange}{RGB}{210,110,35}
\definecolor{softorange}{RGB}{255,241,228}
\definecolor{textdark}{RGB}{35,35,35}
\definecolor{softgray}{RGB}{245,247,249}
\definecolor{AutomColor}{RGB}{0,90,180}
\newcommand{\Autom}[1]{\textcolor{black}{#1}}
\begin{document}
\raggedbottom
\setlength{\parskip}{8pt plus 1pt minus 1pt}
\begin{frontmatter}

\title{ Certificate-based Synthesis of Coordinated Droop Control for Heterogeneous Radial Distribution Networks \thanksref{footnoteinfo}} 
\thanks[footnoteinfo]{*This work is funded through the CETPartnership’s ProRES project under the Joint Call 2024, co‑funded by the European Commission (Grant Agreement No. 101069750) and the Dutch Research Council NWO (File No. EP.1602.24.001).}
\thanks{\mbox{Email: {\small
\texttt{\{G.Pantazis33,M.S.T.Chong\}@tue.nl}.}}}
\author{Georgios Pantazis},    % Add the 
\author{Michelle Chong}             % e-mail address 

\address{Dynamics and Control Section, Department of Mechanical Engineering, Eindhoven University of Technology, the Netherlands}  % Please supply                                              

                             % keyword list or with the 
                                          % help of the Automatica 
                                          % keyword wizard

                         % Abstract of not more than 200 words.
\begin{abstract}
\Autom{Voltage certificates for droop-controlled radial distribution networks are often constructed from worst-case quantities. In heterogeneous radial networks, this approach can hide where  voltage deviations are most likely. Furthermore,  such certificates become increasingly conservative as the network sensitivities accumulate. Leveraging the structure of the linearized DistFlow model and slope-restricted droop controllers, we derive tighter deterministic voltage certificates that retain heterogeneous network, disturbance and inverter characteristics of each bus.  The certificates reveal the buses and local limitations that dominate certified voltage performance. Worst-case bounds are recovered as a special case.  Although tighter, the heterogeneous certificates can still deteriorate downstream due to the network structure. To address this, we leverage our certificates as design variables for coordinated voltage control. Specifically, we develop a virtual droop architecture with coordination and affine feedforward compensation to reshape the effective voltage sensitivity, and formulate a linear program that jointly synthesizes the controller and minimizes its heterogeneous voltage certificates under operational, communication and inverter placement constraints. The resulting controller guarantees the voltage and inverter bounds for all admissible disturbances. Evaluation on two network benchmarks, a five-customer residential feeder and a 26-customer rural network comprising four feeders, demonstrates improved certificates and voltage behaviour under inverter limits.}
\end{abstract}

\end{frontmatter}

\section{Introduction}
The increasing penetration of distributed energy resources (DERs), such as photovoltaic systems (PVs) interfaced with power electronics, has substantially changed  the operation of low voltage (LV) distribution grids in the last decades.  In modern distribution grids, voltage profiles are increasingly affected by variations in active power generation and load consumption \cite{chong2019local}. These effects are especially pronounced in radial distribution networks, where voltage deviations accumulate and, thus,  downstream buses can experience larger voltage deviations than buses close to the substation. A widely used mechanism for inverter-based voltage regulation is droop control, where the reactive power command is adjusted as a function of the locally measured voltage deviation. Droop controllers are attractive because of their simplicity, scalability, and limited reliance on communication. In practice, however, inverters in a distribution feeder are heterogeneous, i.e.,  they may have different time constants, reactive power constraints and capacities, and disturbance levels. Moreover, the position of each inverter in the network topology can affect its voltage profile. Consequently, voltage certificates obtained using worst-case quantities across all nodes can be overly conservative. Specifically, by replacing parameters, different per bus, with global worst-case constants, such certificates may obscure the spatial structure of the grid and lead to conservative controller design, particularly for networks with nonuniform impedances, loads, and inverter capabilities. This motivates the need for certificates that preserve the heterogeneous structure of the network, which can then provide more informative safety guarantees and reveal where conservatism arises.

Droop Control, also known in the literature as local Volt/VAR control, has been studied in terms of equilibrium and
stability interpretations in
\cite{Farivar2013,Zhu2016,zhou2021reverse}. More recent approaches
adapt droop parameters and/or
learn local control actions
\cite{Yuan2024,Cheng2023,Ye2024,Glover2025}.
Safe reinforcement learning methods additionally impose structural
conditions on learned policies to guarantee closed-loop stability
\cite{Cui2022,Feng2024}, primarily seeking stable efficient
equilibria. Communication-based and feedback optimization methods use
online measurements and information exchange to reach or track
feasible, optimal, and/or fair operating points
\cite{Bolognani2019,Gan2016,Bernstein2019,Magnusson2019,Qu2020,Sun2022}. Coordination based on learning methods has also been
combined with safe gradient flows, and
online safety layers
\cite{Feng2023,Chen2023,Hou2025}. 

Robust affine decision rules address uncertainty in Volt/VAR control and have been considered in~\cite{Nazir2023,Shi2024}.
\Autom{Related offline approaches have designed droop control laws taking into consideration the network structure of feeder models
with stability constraints \cite{Baker2018,Murzakhanov2024}.}
In our setting, communication instead enables  a dynamic coordination law which is synthesized offline, i.e.,  no repeated online optimization or learning is required, and the resulting certificate simultaneously bounds the virtual states, voltages, and inverter commands for all times. Related approaches use barrier certificates and robust control for microgrids~\cite{Kundu2019,Bouvier2022}. These works consider different microgrid dynamics and use sum-of-squares or invariant set computation, while online learning under uncertain grid topology is considered in~\cite{Yeh2024}. For the radial network model developed in~\cite{chong2019local,Chong2020}, existing deterministic results, valid at each time, certify a fixed local droop law using aggregated disturbance and controller bounds. This leaves open, for the considered network model, the use of heterogeneous deterministic  certificates as design variables for jointly synthesizing a dynamic coordinated controller under physical and communication constraints.

\subsection{Main Contributions}

Our main contributions with respect to the cognate literature are as follows:

\textbf{(1) Heterogeneous all-time certificates.}
\Autom{For slope-restricted droop controllers, we derive disturbance margins
and voltage certificates for each bus, while retaining the local voltage
sensitivities, disturbance bounds, droop slopes, and inverter limits}. In contrast to the worst-case
certificate in \cite{chong2019local,Chong2020}, the proposed result
preserves the nodal information and recovers the homogeneous
certificate as a special case. This avoids combining worst-case
quantities attained at different buses and reveals which locations,
disturbances, or inverter limits determine the certified performance. \Autom{This result is interesting per se. Simultaneously, it serves a particular functionality as} the obtained certificate is significantly more informative for controller design
and for assessing where additional reactive power capability would
be most valuable.

\textbf{(2) Certificate-based coordinated control.}
We introduce a virtual droop architecture in which a coordination
matrix reshapes the effective voltage sensitivity, and affine
feedforward compensation reduces the residual effect of bounded
disturbances. Forward invariant state bounds provide 
 certificates on the virtual states, voltages, and reactive power commands. To circumvent the bilinear coupling between the coordination matrix gains and invariant set radii for each voltage, we then introduce a  matrix \Autom{transformation} that allows us to recast the synthesis problem as a
linear program that jointly designs the controller, certificates, and invariant state bounds under prescribed
communication and inverter constraints. Learning-based controller designs
\cite{Yuan2024,Cui2022,Feng2024} enforce closed-loop stability
through structural restrictions on the learned policies. \Autom{Our result
addresses a different guarantee, that is deterministic
certificates on the voltage, virtual state, and inverter command at each bus,} which hold for all time under
bounded disturbances. 

Compared to online feedback optimization
\cite{Bernstein2019,Magnusson2019,Qu2020} and robust affine
Volt/VAR designs \cite{Nazir2023,Shi2024}, our method exploits the structure of a widely used radial network model \cite{Baran_Wu} and incorporates the improved heterogeneous certificates (Theorem \ref{thm:heterogeneous-certificate-agent}) as design variables to jointly synthesize a dynamic controller with fixed parameters, feedforward compensation, and invariant state bounds offline (Theorem \ref{thm:neighborhood_sparse_synthesis}). Any solution then certifies the
voltage and inverter command trajectories at each bus under all
admissible disturbances. Online implementation requires only the prescribed communication pattern and local droop dynamics, without
optimization iterations. \Autom{Finally, we extend our framework to account for the more general case, where only a  subset of buses is equipped with controllable inverters, while still certifying the voltages of all monitored buses.}

\section{ Problem Formulation}
\subsection{Notation}
Let \(\mathbb R^N_{\ge 0}\) and \(\mathbb R^N_{>0}\) denote the
nonnegative and positive orthants of \(\mathbb R^N\), respectively.
For vectors \(x,y\in\mathbb R^N\), inequalities such as
\(x\le y\) and \(x<y\) are understood componentwise. The vector of all
ones is denoted by \(\mathbf 1\), and \(\operatorname{diag}(x)\)
denotes the diagonal matrix with diagonal entries given by the vector
\(x\). For a vector \(x\), \(|x|\) denotes the componentwise absolute
value. For a matrix \(A\), \(|A|\) denotes the elementwise absolute
value, and \(A_{ij}\) denotes its \((i,j)\)-th entry. The spectral
radius of a square matrix \(A\) is defined as $\varrho(A):=\max\{|\lambda|:\det(\lambda I-A)=0)\}$. The induced infinity norm
is denoted by \(\|A\|_\infty\). For a symmetric matrix \(A\), the
relations \(A\succ0\), \(A\succeq0\), \(A\prec0\), and \(A\preceq0\)
denote positive definite, positive semidefinite, negative definite, and
negative semidefinite, respectively. We denote the complement of a set $C_i$ by $C_i^c$.

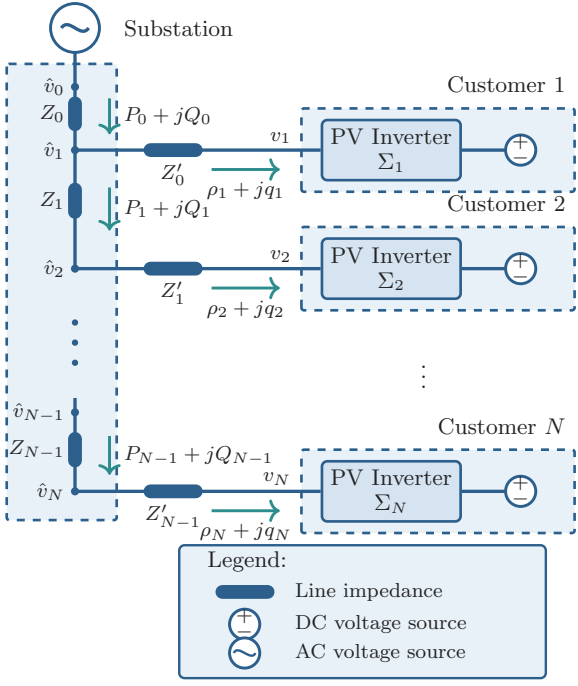
\begin{figure}[t]
\centering
\begingroup

\def\prosumerrow#1#2#3#4#5{%
\begin{scope}[shift={(0,#2)}]

    % Customer dashed box
    \draw[dline, fill=softblue]
        (3.15,-0.58) rectangle (6.95,0.58);

    \node[
        anchor=south east,
        text=textdark
    ] at (6.95,0.68) {Customer $#1$};

    % Inverter
    \node[box] (inv#1) at (4.40,0)
        {PV Inverter\\$\Sigma_{#1}$};

    % Branch from feeder to inverter
    \draw[line] (0,0) -- (inv#1.west);

    % Branch impedance
    \draw[imp] (1.05,0) -- (1.68,0);

    % Branch impedance label
    \node[
        small,
        anchor=north,
        text=textdark
    ] at (1.365,-0.14) {$#3$};

    % Customer power arrow
    \draw[arr] (1.90,-0.27) -- (2.85,-0.27);

    \node[
        small,
        anchor=north,
        text=textdark
    ] at (2.375,-0.40) {$#5$};

    % Local voltage label outside the customer box
    \node[
        small,
        anchor=south east,
        text=textdark
    ] at (3.05,0.05) {$#4$};

    % Wire from inverter to DC source
    \draw[line] (inv#1.east) -- (5.98,0);

    % DC source
    \draw[source] (6.20,0) circle (0.21);
    \node[tinylabel, text=textdark] at (6.20,0.11) {$+$};
    \node[tinylabel, text=textdark] at (6.20,-0.11) {$-$};

\end{scope}
}

\begin{tikzpicture}[
    scale=0.95,
    every node/.style={transform shape},
    font=\footnotesize,
    line/.style={
        draw=gridblue,
        line width=1.25pt
    },
    dline/.style={
        draw=gridblue,
        dashed,
        line width=1.0pt,
        rounded corners=1pt
    },
    arr/.style={
        draw=accentteal,
        ->,
        line width=1.1pt
    },
    imp/.style={
        draw=gridblue,
        line width=5.2pt,
        line cap=round
    },
    source/.style={
        draw=gridblue,
        line width=1.1pt,
        fill=white
    },
    box/.style={
        draw=gridblue,
        fill=boxblue,
        line width=1.0pt,
        rounded corners=2pt,
        minimum width=1.65cm,
        minimum height=0.78cm,
        align=center,
        text=textdark
    },
    small/.style={
        font=\scriptsize,
        inner sep=1pt
    },
    tinylabel/.style={
        font=\scriptsize,
        inner sep=0.5pt
    }
]

% =========================================================
% Distribution grid
% =========================================================
\draw[dline, fill=softblue]
    (-0.95,-0.50) rectangle (0.58,-6.85);
% Substation
\draw[source, fill=white] (0,0) circle (0.34);

\draw[line]
    (-0.19,0)
    .. controls (-0.08,0.16) and (0.08,-0.16) ..
    (0.19,0);

\node[
    anchor=west,
    text=textdark
] at (0.56,0) {Substation};

% =========================================================
% Main feeder
% =========================================================

% Upper feeder section
\draw[line] (0,-0.34) -- (0,-3.35);

% Lower feeder section:
% starts above v'_{N-1}, ends exactly at v'_N
\draw[line] (0,-5.15) -- (0,-6.45);

% Three dots representing omitted feeder sections
\fill[gridblue] (0,-4.10) circle (1.25pt);
\fill[gridblue] (0,-4.38) circle (1.25pt);
\fill[gridblue] (0,-4.66) circle (1.25pt);

% =========================================================
% Feeder voltage labels
% =========================================================

% Voltage measurement between the substation and Z_0
\node[
    small,
    anchor=east,
    text=textdark
] at (-0.10,-0.82) {$\hat v_0$};

% Customer connection voltages
\node[
    small,
    anchor=east,
    text=textdark
] at (-0.10,-1.70) {$\hat v_1$};

\node[
    small,
    anchor=east,
    text=textdark
] at (-0.10,-3.35) {$\hat v_2$};

% Voltage measurement before Z_{N-1}
\node[
    small,
    anchor=east,
    text=textdark
] at (-0.10,-5.35) {$\hat v_{N-1}$};

% Final customer node
\node[
    small,
    anchor=east,
    text=textdark
] at (-0.10,-6.45) {$\hat v_N$};

% =========================================================
% Feeder impedances
% =========================================================

% Impedance Z_0
\draw[imp] (0,-1.05) -- (0,-1.34);

\node[
    small,
    anchor=east,
    text=textdark
] at (-0.10,-1.195) {$Z_0$};

% Impedance Z_1
\draw[imp] (0,-2.25) -- (0,-2.56);

\node[
    small,
    anchor=east,
    text=textdark
] at (-0.10,-2.405) {$Z_1$};

% Final feeder impedance Z_{N-1}
\draw[imp] (0,-5.72) -- (0,-6.02);

\node[
    small,
    anchor=east,
    text=textdark
] at (-0.10,-5.87) {$Z_{N-1}$};

% =========================================================
% Feeder power-flow arrows
% =========================================================

% Power flow through Z_0
\draw[arr] (0.48,-0.98) -- (0.48,-1.48);

\node[
    small,
    anchor=west,
    text=textdark
] at (0.64,-1.23) {$P_0+jQ_0$};

% Power flow through Z_1
\draw[arr] (0.48,-2.22) -- (0.48,-2.82);

\node[
    small,
    anchor=west,
    text=textdark
] at (0.64,-2.52) {$P_1+jQ_1$};

% Power flow through Z_{N-1}
\draw[arr] (0.48,-5.68) -- (0.48,-6.22);

\node[
    small,
    anchor=west,
    text=textdark
] at (0.64,-5.95) {$P_{N-1}+jQ_{N-1}$};

% =========================================================
% Connection and voltage measurement points
% =========================================================

% Measurement point for v'_0
\fill[gridblue] (0,-0.82) circle (1.6pt);

% Customer connection points
\fill[gridblue] (0,-1.70) circle (1.6pt);
\fill[gridblue] (0,-3.35) circle (1.6pt);

% Measurement point for v'_{N-1}
\fill[gridblue] (0,-5.35) circle (1.6pt);

% Final node v'_N
\fill[gridblue] (0,-6.45) circle (1.6pt);

% =========================================================
% Customer rows
% =========================================================

% First customer moved downward to lengthen the section
% between the substation and the first customer
\prosumerrow{1}{-1.70}{Z'_0}{v_1}{\rho_1+jq_1}

\prosumerrow{2}{-3.35}{Z'_1}{v_2}{\rho_2+jq_2}

% Omitted middle customers
\node[text=textdark] at (4.85,-4.75) {$\vdots$};

% Last customer
\prosumerrow{N}{-6.45}{Z'_{N-1}}{v_N}{\rho_N+jq_N}

% =========================================================
% Legend
% =========================================================
\draw[
    draw=gridblue,
    fill=softblue,
    line width=0.9pt,
    rounded corners=2pt
] (1.45,-9.05) rectangle (6.55,-7.20);

\node[
    anchor=west,
    text=textdark
] at (1.72,-7.42) {Legend:};

% Line impedance
\draw[imp] (2.05,-7.85) -- (2.68,-7.85);

\node[
    small,
    anchor=west,
    text=textdark
] at (3.02,-7.85) {Line impedance};

% DC voltage source
\draw[source] (2.36,-8.30) circle (0.21);

\node[tinylabel, text=textdark]
    at (2.36,-8.19) {$+$};

\node[tinylabel, text=textdark]
    at (2.36,-8.41) {$-$};

\node[
    small,
    anchor=west,
    text=textdark
] at (3.02,-8.30) {DC voltage source};

% AC voltage source
\draw[source] (2.36,-8.70) circle (0.21);

\draw[
    draw=gridblue,
    line width=0.9pt
]
    (2.21,-8.70)
    .. controls (2.27,-8.57) and (2.45,-8.83) ..
    (2.51,-8.70);

\node[
    small,
    anchor=west,
    text=textdark
] at (3.02,-8.70) {AC voltage source};

\end{tikzpicture}
\endgroup

\caption{Radial network with \(N\) prosumers. The net complex power injection of customer \(i\)
is
\(\rho_i+\mathrm{j}q_i\), where
\(\rho_i=\rho_{g,i}-\rho_{c,i}\) and
\(q_i=q_{g,i}-q_{c,i}\).
Along the feeder, the complex power \(P_i+\mathrm{j}Q_i\) flows through
\(Z_i=R_i+\mathrm{j}X_i\) towards connection point \(i+1\).}
\label{fig:lv-distribution-grid}
\end{figure}

\subsection{Radial Network Model}
We consider a model of a radial low voltage (LV) distribution grid with $N$ prosumers based on the Baran-Wu model, also known as DistFlow equations \cite{Baran_Wu}. Each prosumer $i \in \mathcal{N}=\{1, \dots, N\}$ is equipped with an inverter that can inject active and reactive power $\rho_{g, i}$ and $q_{g,i}$, respectively, based on renewable energy generation such as photovoltaics. Let $v_i$ denote the voltage magnitude at customer $i$, and $\hat{v}_i$ the voltage magnitude at the corresponding connection point on the feeder.  The feeder segment between connection points \(i\) and \(i+1\) has impedance $Z_i = R_i + j X_i, i \in  \mathcal{N}_{-}=\{0, \dots, N-1\}$, while the impedance of the line between customer $i$ and its connection point is $Z'_{i-1} = R'_{i-1} + j X'_{i-1}, i \in \mathcal{N}$. Let $P_i$ and $Q_i$ denote the active and reactive power flowing in the feeder from connection point $i$ to the connection point of the next customer $i+1$. For each customer, define the net active power injection $\rho_i := \rho_{g,i} - \rho_{c,i}$,
and the net reactive power injection $q_i := q_{g,i} - q_{c,i}$. Since losses are typically much smaller than power flows $P_i$ and $Q_i$, the relative error is small (typically of the order of  1\%) \cite{Baran_Wu}. Hence, the feeder satisfies the following power flow equations:
\[
\left\{
\begin{aligned}
    P_{i+1} &= P_i + \rho_{i+1}, \\
    Q_{i+1} &= Q_i + q_{i+1}, \\
    \hat v_{i+1}^2 &= \hat v_i^2 - 2\beta_i(P_i,Q_i), \\
    \hat v_i^2 &= v_i^2 - 2\beta'_{i-1}(\rho_i,q_i),
\end{aligned}
\right.
\]
for $i \in \mathcal{N}_{-}$, where $\beta_i(r,s) := R_i r + X_i s$, $\beta'_i(r,s) := R'_i r + X'_i s$ and setting  $\beta'_{-1}(\cdot,\cdot) \equiv 0$. The voltage at each customer $i \in \mathcal{N}$ is regulated by a local droop control through the inverter dynamics:
\begin{equation}
    \dot q_{g,i}(t)
    =
    -\frac{1}{\tau_i} q_{g,i}(t)
    +
    \frac{1}{\tau_i} K_i\!\bigl(\bar v^2 - v_i^2(t)\bigr), \nonumber 
\end{equation}
where \(\tau_i > 0\) is the time constant of each inverter $i$, $\bar{v}$ is the nominal voltage reference, and $K_i$ is the local droop nonlinearity which can be used to control the voltage through the reactive power. We impose the following assumption on the droop function: \\
\begin{assumption} \label{assum:sector-bounded} 
1) The droop function $K(y)=(K_1(y_1), \dots, K_N(y_N))^\top$ satisfies
\begin{align}
\frac{K_i(w)-K_i(v)}{w-v} \in [0, d_i], \nonumber 
\end{align}
for all $w,v \in \mathbb{R}$ with $w \neq v$ 
and $K_i(0)=0$ for each $i \in \mathcal{N}$. \\
2) The bound $|K_i(v)| \leq \bar{K}_i$ holds  for all $v \in \mathbb{R}$ and each $i \in \mathcal{N}$. \hfill $\square$
\end{assumption}
Assumption \ref{assum:sector-bounded} is standard in droop  control implementations. 
The sector boundedness captures the monotone response of the reactive power with a maximum slope \(d_i\), while the bound \(|K_i(v)|\le \overline K_i\) reflects the reactive power capacity limit of inverter \(i\). 

A commonly used example satisfying Assumption~\ref{assum:sector-bounded} is the piecewise-affine droop function used in
\cite{chong2019local,zhou2021reverse}. Similar to \cite{chong2019local}, we choose the design parameters such that $y_{\min,i}<y_{m,i}\leq 0\leq y_{n,i}<y_{\max,i}$.
The droop function is then defined as
\[
K_i(y)=
\begin{cases}
-\bar K_i,
    & y\leq y_{\min,i},\\[1mm]
-\left(
1-\dfrac{y-y_{\min,i}}
{y_{m,i}-y_{\min,i}}
\right)\bar K_i,
    & y_{\min,i}<y\leq y_{m,i},\\[3mm]
0,
    & y_{m,i}<y\leq y_{n,i},\\[1mm]
\dfrac{y-y_{n,i}}
{y_{\max,i}-y_{n,i}}\bar K_i,
    & y_{n,i}<y\leq y_{\max,i},\\[3mm]
\bar K_i,
    & y>y_{\max,i}.
\end{cases}
\]
The interval $(y_{m,i},y_{n,i}]$ is the so-called dead zone, within which no
reactive power response is requested. Outside the dead zone, the droop
output varies affinely with the voltage error until it reaches the
saturation values $-\bar K_i$ and $\bar K_i$. When the controller is
implemented at customer $i$, the scalar argument $y$ is evaluated at
$y=y_i=\bar v^2-v_i^2$.
Consequently, the breakpoint parameters $y_{\min,i}$, $y_{m,i}$,
$y_{n,i}$, and $y_{\max,i}$ are expressed in squared voltage units. 

When the function $K_i$ is slope restricted (Assumption~\ref{assum:sector-bounded}-1), it implies that
$\bar{K}_i \leq d_i(y_{m,i}-y_{\min,i})$ and
$\bar{K}_i \leq d_i(y_{\max,i}-y_{n,i})$. The saturation value $\bar K_i$ represents the reactive power capability
allocated to voltage regulation and is therefore constrained by the
inverter rating. In particular, if $\bar s_i$ is the
apparent power rating of inverter $i$ and $\rho_{g,i}(t)$ is its active power output at time $t$, the inverter capability imposes $\rho_{g,i}^2(t)+q_{g,i}^2(t)\leq \bar s_i^2$,
and hence $|q_{g,i}(t)|
\leq
\sqrt{\bar s_i^2-\rho_{g,i}^2(t)}$.
Accordingly, the saturation value must satisfy the condition
$\bar K_i
\leq
\sqrt{\bar s_i^2-\rho_{g,i}^2(t)}$.
For a time-invariant bound as the one used in Assumption~\ref{assum:sector-bounded}, $\bar K_i$ should not
exceed the minimum available reactive power headroom over the admissible
active power range. For example, if $\max_{t \geq 0}|\rho_{g,i}(t)|= \rho_{g,i}^{\max}$,
one may select $\bar K_i
\leq
\sqrt{
\bar s_i^2-
\left(\rho_{g,i}^{\max}\right)^2
}$. Consequently, the available reactive power decreases as the active power approaches the inverter rating. A nonzero reserve at peak generation can be ensured by selecting an inverter rating above the maximum active power output \cite{zhou2021reverse,Turitsyn2011}.

We now define for each customer  $i \in \mathcal{N}$ the squared voltage differences $y_i := \bar v^2 - v_i^2$ and $y_0 := \bar v^2 - \hat v_0^2$,
where $\hat{v}_0$ is the voltage on the head of the distribution grid. Furthermore, we define
    $\eta_i := v_{i-1}^2 - v_i^2$,
with \(v_0 := \hat v_0\). Then, it can  be shown that 
\begin{align} \label{eq:recursion}
y_i = y_{i-1} + \eta_i.   
\end{align}
Using the equations above, it holds that:
\begin{align}
    &\eta_i
    =
    (v_{i-1}^2 - \hat v_i^2) + (\hat v_i^2 - v_i^2) \nonumber \\
    &=
    2\beta_{i-1}(P_{i-1},Q_{i-1})
    +
    2\beta'_{i-2}(\rho_{i-1},q_{i-1})
    -
    2\beta'_{i-1}(\rho_i,q_i). \nonumber 
\end{align}
Hence the closed-loop system is equivalently described by the dynamics:
\[
\left\{
\begin{aligned}
    \dot q_{g,i}
    &=
    -\frac{1}{\tau_i} q_{g,i}
    +
    \frac{1}{\tau_i} K_i(y_i), \\
    y_i &= y_{i-1}+\eta_i .
\end{aligned}
\right.
\]
Let us introduce the vectors  $q_g := (q_{g,1},\dots,q_{g,N})^\top$, and $ 
y:= (y_1,\dots,y_N)^\top$. 
 Imposing $P_N = Q_N = 0$, by assuming that the network ends after customer $N$,
and solving recursively  for \(P_i\) and \(Q_i\) we obtain for each node $i \in \mathcal{N}$: 
\begin{align} \label{eq:Pi}
    P_i &= -\sum_{j=i+1}^{N} \rho_j, \ Q_i = \sum_{j=i+1}^{N} (q_{c,j} - q_{g,j}), 
\end{align}
Substituting these expressions into \(\eta_i\) yields:
\begin{align}
    \eta_i
    &=
    -2R_{i-1}\sum_{j=i}^{N}\rho_j
    +
    2X_{i-1}\sum_{j=i}^{N}(q_{c,j}-q_{g,j}) \nonumber\\
    &\quad
    +
    2\beta'_{i-2}(\rho_{i-1},q_{i-1})
    -
    2\beta'_{i-1}(\rho_i,q_i). \nonumber 
\end{align}

Separating the controllable term $q_g$ from the exogenous terms $(\rho,q_c)$ and $y_0$ and then summing the recursion $y_i = y_0 + \sum_{k=1}^{i}\eta_k$, for each $i \in \mathcal{N}$,
the output dynamics take the form:
\begin{equation}
    y_i
    = (H q_{g})_i
    +
    [\phi(\rho,q_c)]_i
    +
    y_0, \nonumber 
\end{equation}
where $(Hq_g)_i \in \mathbb{R}$ denotes the $i$-th component of the vector $Hq_g \in \mathbb{R}^N$
and 
 $H=(h_{ij})_{i,j\in\mathcal N}$ is defined as:
\begin{align}
H
=
-2
\begin{bmatrix}
X_0 & X_0 & \cdots & X_0 \\
X_0 & X_0+X_1 & \cdots & X_0+X_1 \\
\vdots & \vdots & \ddots & \vdots \\
X_0 & X_0+X_1 & \cdots & \displaystyle\sum_{k=0}^{N-1}X_k \nonumber 
\end{bmatrix} \\
-
2\operatorname{diag}(X'_0,\dots,X'_{N-1}). \nonumber 
\end{align}

The term \(\phi(\rho,q_c)\) collects all contributions of net active power injections and consumptions.  Writing the derivation of $\phi$ in \cite{chong2019local} in analytical form, its \(i\)-th component can be written as
\begin{align}
    [\phi(\rho,q_c)]_i
    =
    \sum_{k=0}^{i-1}
    \left(
        2X_k \sum_{j=k+1}^{N} q_{c,j}
        -
        2R_k \sum_{j=k+1}^{N} \rho_j
    \right) \nonumber \\
    -2R'_{i-1}\rho_i
+2X'_{i-1}q_{c,i}. \nonumber 
\end{align}
 Finally, defining the diagonal matrix $T := \mathrm{diag}(\tau_1,\dots,\tau_N)$,
the closed-loop system takes the compact form:
\begin{align}
    \dot q_g &= -T^{-1} q_g + T^{-1} K(y), \label{state_dynamics} \\
    y &= Hq_g + \phi(\rho,q_c) + y_0 \mathbf{1}, \label{output_dynamics}
\end{align}
where  $K(y) := (K_1(y_1),\dots,K_N(y_N))^\top$. Let us now define the matrix $\widetilde H := (\widetilde h_{ij})_{i,j\in\mathcal N}$, where 
$\widetilde h_{ij}:=|h_{ij}|.$ \Autom{In other words, $\widetilde H$ is obtained from $H$ by taking the absolute value of each of its components.}
Since for reactances it always holds \Autom{that} $X_k\ge 0$ and $X'_k\ge 0$, each element $\widetilde h_{ij}$ can be written as:
\[
\widetilde h_{ij}
=
2\sum_{k=0}^{\min\{i,j\}-1}X_k
+
2X'_{i-1}\mathbf 1_{\{i=j\}}, \ i,j\in\mathcal N .
\]

\subsection{Problem Statement}

In this paper, we achieve the following objectives:
\begin{enumerate}
    \item[(i)] First, Section~3 derives heterogeneous voltage certificates for standard local droop control.
    \item[(ii)] Second, Section~4 introduces virtual coordination and affine feedforward compensation and shows how these mechanisms reshape the voltage certificates while preserving forward invariance and inverter feasibility.
    \item[(iii)] Third, Section~5 converts the resulting certificate conditions into a linear synthesis problem under prescribed communication and extends the framework to sparse inverter placement.
\end{enumerate}

\section{Heterogeneous voltage certificates}
In this section, we provide heterogeneous voltage certificates for each customer. To this end, we impose the following assumption: \\

\begin{assumption}
\label{assum:exogenous-bounds}
For each customer \  $i \in \mathcal{N}$, the net active power injection
$\rho_i(t)$ and the reactive power consumption $q_{c,i}(t)$ are
bounded, i.e., there exist
$\bar\rho_i\ge 0$ and $\bar q_{c,i}\ge 0$ such that $|\rho_i(t)|\le \bar\rho_i$, and 
$|q_{c,i}(t)|\le \bar q_{c,i}$ for all $ t\ge 0.$ \hfill $\square$
\end{assumption}

Assumption~\ref{assum:exogenous-bounds} is reasonable in practice.
As discussed in Section~2.2,  the net active power injection
\(\rho_i=\rho_{g,i}-\rho_{c,i}\) and the reactive power consumption
\(q_{c,i}\) admit finite bounds.  The following result shows how one can leverage the model presented in Section 2 to obtain bounds on the exogenous term of each customer.

\begin{lemma}
\label{lem:exogenous-bound-customer-i}
Consider the output dynamics in (\ref{output_dynamics}).
Under Assumption~\ref{assum:exogenous-bounds}, for each $i \in \mathcal{N}$ and for all $t \geq 0$
it holds that
$|\phi_i(\rho(t),q_c(t))|
\le
\bar\phi_i$,
where
\begin{align}
\bar\phi_i
:=  
2\sum_{k=0}^{i-1}
\left(
R_k\sum_{j=k+1}^{N}\bar\rho_j
+
X_k\sum_{j=k+1}^{N}\bar q_{c,j}
\right)+ \nonumber  \\ \nonumber 
+
2R'_{i-1}\bar\rho_i
+
2X'_{i-1}\bar q_{c,i}. \ \ \ \ \ \ \ \  \  \square
\end{align}
\end{lemma}

\emph{Proof:} See Appendix. \hfill $\blacksquare$

Lemma~1 implies that bounded variations in customer power generation and load consumption produce a bounded voltage disturbance at each bus, which is quantifiable through the model. 
In practice, this allows the controller to use a known margin \(\bar{\phi}_i\) per bus. This bound constitutes an improvement over the worst-case bound used in \cite{chong2019local}. We also assume that the deviation of the voltage $\hat{v}_0$ at the head of the line from the desired voltage $\bar{v}$ is bounded, as stated below.
\begin{assumption}
\label{assum:feeder-head-voltage}
There exist constants
\(\underline y_0,\overline y_0\in\mathbb R\), with
\(\underline y_0\leq\overline y_0\), such that
$y_0(t)\in\mathcal Y_0
:=
[\underline y_0,\overline y_0]$,
\text{for all }$t\geq0$. As such, there exists a uniform bound
$\bar{\varepsilon}:=\max\{|\underline{y}_0|,|\overline{y}_0|\}$
such that $|y_0(t)|\leq\bar{\varepsilon}$ for all $t\geq0$.  \hfill $\square$
\end{assumption}
The following result then holds.\\ 

\begin{theorem}
\label{thm:heterogeneous-certificate-agent}
Consider the network model
\eqref{state_dynamics}-\eqref{output_dynamics} under
Assumptions~1-\ref{assum:feeder-head-voltage}.
Then, for every \(q_g(0)\in\mathbb R^N\), each
\(i\in\mathcal N\), and all \(t\geq0\), it holds that
\begin{align}
|y_i(t)|
\leq
\sum_{j=1}^{N}
\widetilde h_{ij}
\max\left\{
|q_{g,j}(0)|,m_j
\right\}
+\ell_i, \label{eq:bound1}
\end{align}
where
$\ell_j:=\bar\phi_j+\bar\varepsilon$, $\bar Q_j:=\max\left\{|q_{g,j}(0)|,\bar K_j\right\}$,$\gamma_j
:=
\sum_{k=1}^{N}\widetilde h_{jk}\bar Q_k+\ell_j$ and 
$m_j:=\min\left\{\bar K_j,d_j\gamma_j\right\}$, for all  $j\in\mathcal N$. \hfill $\square$
\end{theorem}

\emph{Proof:} See Appendix. \hfill $\blacksquare$

Theorem \ref{thm:heterogeneous-certificate-agent} substantially improves upon the homogeneous certificate of Theorem 1 in~\cite{chong2019local}. Specifically, the homogeneous bound is
recovered from the heterogeneous certificate by replacing
the quantities parameterized by the inverter or bus index  $d_j$, $\overline K_j$, $\ell_j$,
and the rows of $\widetilde H$, by their worst case bound, as shown in Corollary \ref{cor:homogeneous-special-case-customer-i}. \\

\begin{corollary}
\label{cor:homogeneous-special-case-customer-i}
Under Assumptions \ref{assum:sector-bounded}-\ref{assum:feeder-head-voltage},
define
$d:=\max_{j\in\mathcal N}d_j$,
$\tau:=\max_{j\in\mathcal N}\tau_j$,
$\tau_m:=\min_{j\in\mathcal N}\tau_j$, and
$\bar K:=\max_{j\in\mathcal N}\bar K_j$. To match the notation of~\cite{chong2019local}, let
\(\bar{\epsilon}_y := \bar{\epsilon}\).
Moreover, let us define the homogeneous upper bounds on all $\phi_i$ and all $y_i$, respectively as
$\Delta_\phi:=\max_{j\in\mathcal N}\bar\phi_j,
\bar\ell:=\Delta_\phi+\bar\epsilon_y$, and let $\|\widetilde H\|_\infty
=
\max\limits_{r\in\mathcal N}\sum\limits_{s=1}^N \widetilde h_{rs}$.
Then, the certificate in Theorem \ref{thm:heterogeneous-certificate-agent} implies the certificate of Theorem 1 in \cite{chong2019local}, i.e., if the inequality
\begin{align}
&
\|\widetilde H\|_\infty
\|q_g(0)\|_\infty
+
\frac{d\tau}{\tau_m}
\|\widetilde H\|_\infty
\bar\ell
\nonumber\\
&\quad+
\frac{d\tau}{\tau_m}
\|\widetilde H\|^2_\infty
\bigl(\|q_g(0)\|_\infty+\bar K\bigr)
+
\bar\ell
\le
\varepsilon, \label{eq:bound2}
\end{align}
is satisfied for a preselected $\varepsilon$,
then $|y_i(t)|\le \varepsilon,$ for all $t\ge 0$. \hfill $\square$
\end{corollary}
\emph{Proof:} See Appendix. \hfill $\blacksquare$

\begin{figure}[t]
    \centering
    \includegraphics[width=0.83\columnwidth]{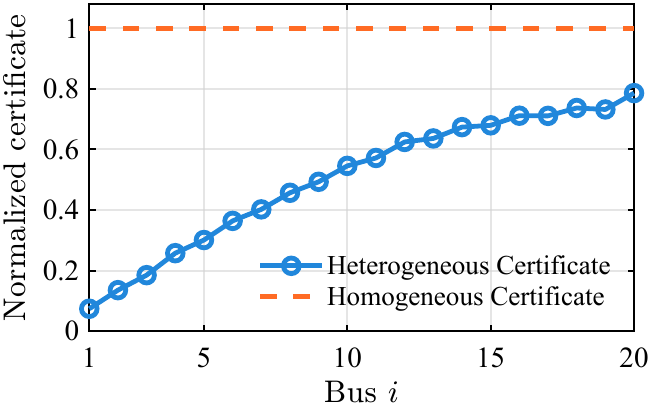}
    \caption{Comparison between the heterogeneous certificate of
    Theorem~1 and the corresponding homogeneous worst-case certificate in Theorem 1 in \cite{chong2019local}.}
    \label{fig:certificate-comparison-section3}
\end{figure}
\textbf{Example 1.} 
Consider the setting of Theorem~\ref{thm:heterogeneous-certificate-agent} for an $N$-bus radial feeder with
identical feeder reactances $X_k=x>0$, while allowing the reactances $X'_{i-1}$, disturbance margins $\ell_i$,
and the quantities $\max\{|q_{g,i}(0)|,m_i\}$ to be bus dependent.
Then, from the definition of $\widetilde H$, it holds that $\widetilde h_{ij}
    =
    2x\min\{i,j\}
    +
    2X'_{i-1}\mathbf 1_{\{i=j\}}.$
Hence Theorem~1 gives the bus dependent certificate in \eqref{eq:bound1}, while the corresponding homogeneous certificate yields the same certificate bound obtained by \eqref{eq:bound2}.
Figure~2 compares these two certificates after normalization by the
homogeneous bound. 

Specifically, for illustrative purposes we consider a single radial feeder with
$N=20$ customers and identical feeder reactances
$X_k=x=2\times10^{-3}\,\Omega$. The customer reactances at the corresponding connection points are
selected within the range $X'_{i-1}\in[0.012,0.080]\ \Omega$,
while the quantities related to the inverter are chosen to satisfy
$\max\{|q_{g,i}(0)|,m_i\}\in[450,1000]\ \mathrm{VAr}$. Finally, the local disturbance margins satisfy
$\ell_i\in[1,5]\ \mathrm{V}^2$.
The parameters are chosen such that the largest row sum of
$\widetilde H$, the largest disturbance margin, and the largest inverter command bound occur at buses $20$, $4$, and $12$, respectively.
Thus, the homogeneous certificate combines worst-case quantities that
are attained at different locations. The resulting homogeneous bound is
$b^{\mathrm{hom}}=1005\ \mathrm{V}^2$, and both curves in
Figure~\ref{fig:certificate-comparison-section3} are normalized by this
value.

The heterogeneous certificate of Theorem 1 preserves the heterogeneity in inverter dynamics, droop slopes, reactive power limits,
and disturbances on the customer side. Although it is less conservative than
the homogeneous certificate,  the certificate
still inherits the topological limitations of radial feeders, i.e., 
voltage deviations accumulate along the feeder through the sensitivity
matrix $H$. Consequently, customers located farther downstream may
receive larger certified voltage bounds, even when their local
disturbance and controller parameters are not more adverse. This motivates a controller design whose purpose is not only to reduce
the magnitude of the voltage certificate, but also to make the
certificate more uniform across customers.

\section{Coordinated Control Design}
\subsection{Certificate reshaping through virtual coordination}
We now consider the following virtual controller:
\begin{equation} \label{virtual_controller}
\left\{
\begin{aligned}
q_g(t) &= Gz(t) + \tilde{q}(t), \\
\dot z(t) &= -T^{-1}z(t) + T^{-1}K(y(t)), 
\end{aligned}
\right.
\end{equation}
where \(G\in\mathbb R^{N\times N}\) is a  coordination matrix,
\(z\in\mathbb R^N\) is the virtual droop state vector, and \(\tilde{q}(t)\) is a
feedforward controller to be designed. For notational convenience, define the aggregate exogenous voltage
disturbance $\delta(t)
:=
\phi\bigl(\rho(t),q_c(t)\bigr)+y_0(t)\mathbf 1.
$
Then, the output dynamics in \eqref{output_dynamics} can be written as $y(t)=Hq_g(t)+\delta(t)$. Based on the virtual controller (\ref{virtual_controller}), the output dynamics (\ref{output_dynamics}) take the form: 
\begin{align} \label{new_output}
y=F z(t)+r(t),
\end{align}
where $F=HG$ and $r(t)=H\tilde{q}(t)+\delta(t)$. We denote by $F_{ij}$ the element of the $i$-th row and $j$-th column of matrix $F$. The controller in \eqref{virtual_controller} should be interpreted as
a coordinated extension of local droop control by leveraging a coordination graph over the inverters described by $G$. This virtual controller can be used as a purely
local droop law in the special case when $G$ is chosen to be diagonal. Note that the nonlinear feedback component in the second part of the controller remains local in the
sense that each virtual droop state $z_i$ is driven by the local voltage
deviation $y_i$.  

%The purpose of this section is to quantify the voltage certificates that can
%be obtained when coordination of inverters on the feeder level is
%available. The matrix $G$ can be fixed or tuned as a solution to an optimization problem (as shown in the developments of the current section and Section 5),
%while $\tilde{q}$ can be tuned in a similar manner on the basis of measurements, forecasts, or disturbance envelopes. Figure~\ref{fig:coordinated_droop_scheme} illustrates the signal flow of the coordinated droop controller
%in~\eqref{virtual_controller}, showing how local voltage feedback, virtual coordination, and
%feedforward disturbance compensation are interconnected. 
Motivated by the downstream deterioration identified in
Section~3, we use the coordination matrix $G$ to reshape the voltage sensitivity of the network. We first consider
the case where $G$ is fixed, while Section~5 develops its joint
synthesis with the voltage certificates. Figure~\ref{fig:coordinated_droop_scheme} illustrates how
the local droop dynamics, virtual coordination, and feedforward
compensation interact.

For the case when $G$ is fixed by the designer and $H$ is invertible, we can choose
\begin{align}
G=H^{-1}W, \label{eq:G}
\end{align}
where $W$ is a preselected diagonal matrix. A simple choice of such a matrix is $W=-\alpha I$, with $\alpha$ chosen by the designer. This yields the output dynamics of the form: 
\begin{align}
y(t)=-\alpha z(t) + H \tilde{q}(t)+\delta(t).  \nonumber 
\end{align}

To show the invertibility of matrix $H$, we derive the following result. 

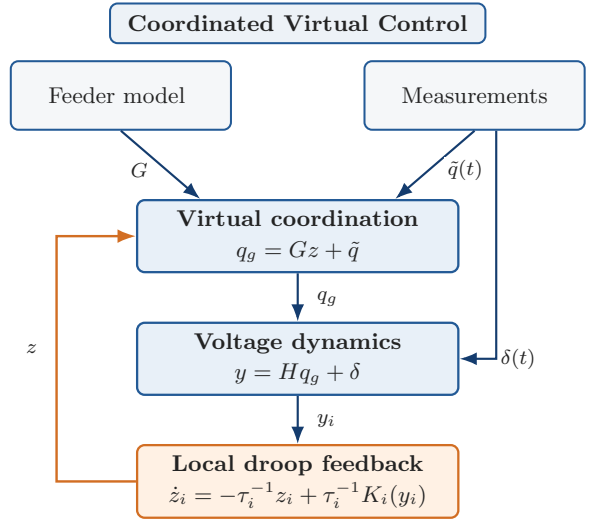
\begin{figure}[t]
\centering
\makebox[\columnwidth][c]{%
\resizebox{0.95\columnwidth}{!}{%
\begin{tikzpicture}[
    >=Latex,
    font=\small,
    every node/.style={text=textdark},
    titleblock/.style={
        draw=controlblue,
        fill=controlblue!8,
        rounded corners=3pt,
        line width=0.8pt,
        align=center,
        minimum width=5.2cm,
        minimum height=0.55cm,
        inner sep=3pt
    },
    mainblock/.style={
        draw=controlblue,
        fill=softblue,
        rounded corners=3pt,
        line width=0.9pt,
        align=center,
        minimum width=4.45cm,
        minimum height=0.95cm,
        inner sep=4pt
    },
    inputblock/.style={
        draw=controlblue,
        fill=softgray,
        rounded corners=3pt,
        line width=0.8pt,
        align=center,
        minimum width=3.05cm,
        minimum height=0.95cm,
        inner sep=3pt
    },
    feedbackblock/.style={
        draw=signalorange,
        fill=softorange,
        rounded corners=3pt,
        line width=0.9pt,
        align=center,
        minimum width=4.45cm,
        minimum height=0.95cm,
        inner sep=4pt
    },
    signal/.style={
        ->,
        draw=darkblue,
        line width=0.85pt
    },
    feedback/.style={
        ->,
        draw=signalorange,
        line width=0.95pt
    },
    siglab/.style={
        font=\scriptsize,
        fill=white,
        inner sep=1.2pt,
        text=textdark
    }
]

% ---------------------------------------------------------
% Bounding box for stable centering
% ---------------------------------------------------------
\path[use as bounding box] (-4.2,1.35) rectangle (4.2,-6.35);

% ---------------------------------------------------------
% Title
% ---------------------------------------------------------
\node[titleblock] (title) at (0,1.00) {%
    \textbf{Coordinated Virtual Control}%
};

% ---------------------------------------------------------
% Upper information blocks
% ---------------------------------------------------------
\node[inputblock] (model) at (-2.45,0) {%
    Feeder model
};

\node[inputblock] (dist) at (2.45,0) {%
    Measurements%
};

% ---------------------------------------------------------
% Main vertical chain
% ---------------------------------------------------------
\node[mainblock] (coord) at (0,-1.95) {%
    \textbf{Virtual coordination}\\[0.7mm]
    $q_g = Gz + \tilde{q}$%
};

\node[mainblock] (feeder) at (0,-3.65) {%
    \textbf{Voltage dynamics}\\[0.7mm]
    $y = Hq_g + \delta$%
};

\node[feedbackblock] (droop) at (0,-5.35) {%
    \textbf{Local droop feedback}\\[0.7mm]
    $\dot z_i = -\tau_i^{-1}z_i + \tau_i^{-1}K_i(y_i)$%
};

% ---------------------------------------------------------
% Input arrows with labels BESIDE the arrows
% ---------------------------------------------------------
\draw[signal]
    (model.south) -- ($(coord.north)+(-1.35,0)$)
    node[pos=0.58, left=6pt, siglab] {$G$};

\draw[signal]
    (dist.south) -- ($(coord.north)+(1.35,0)$)
    node[pos=0.58, right=6pt, siglab] {$\tilde{q}(t)$};

% Disturbance / residual input to feeder
\draw[signal]
    ($(dist.south)+(0.30,0)$) |- (feeder.east)
    node[pos=0.70, right=6pt, siglab] {$\delta(t)$};

% ---------------------------------------------------------
% Forward path
% ---------------------------------------------------------
\draw[signal]
    (coord.south) -- (feeder.north)
    node[pos=0.52, right=6pt, siglab] {$q_g$};

\draw[signal]
    (feeder.south) -- (droop.north)
    node[pos=0.52, right=6pt, siglab] {$y_i$};

% ---------------------------------------------------------
% Feedback path
% ---------------------------------------------------------
\draw[feedback]
    (droop.west) -- ++(-1.10,0) |- (coord.west)
    node[pos=0.27, left=6pt, siglab] {$z$};

\end{tikzpicture}%
}%
}
\caption{Coordinated Virtual Control Flowchart}
\label{fig:coordinated_droop_scheme}
\end{figure}

\begin{lemma} \label{lem:invertible}
If $X_i>0$ and $X'_i \geq 0$ for all $i \in \mathcal{N}_{-}$, then $H$ is negative definite.  \hfill $\square$
\end{lemma}
\emph{Proof:} See Appendix.  \hfill $\blacksquare$

An immediate consequence of Lemma \ref{lem:invertible} is that \(H\) is also invertible, hence a coordination matrix in relation \eqref{eq:G} can indeed be used to alleviate the downstream deterioration of the certificates offered in Section 3. 
For a fixed $G$ we now need to select the feedforward controller $\tilde{q}$. In contrast to Section 3, the aim of this section is not merely to certify a controller, but to use the certificates themselves to drive their design via the formulation of an appropriate optimization problem. We next derive conditions ensuring forward invariance of the
virtual state set. To this end, we impose the following assumptions:
\begin{comment}
To achieve this, we define the notion of forward invariance.

\begin{definition}
A set $\mathcal S\subseteq\mathbb R^N$ is called forward invariant for
the system $\dot x=f(x,t)$ if every solution starting in $\mathcal S$
remains in $\mathcal S$ for all future time. That is, if $x(0)\in\mathcal S$,
then $x(t)\in\mathcal S, \forall t\ge 0$. \hfill $\square$
\end{definition} 
\end{comment}
%We now impose the following assumptions:
\begin{assumption} \label{assum:res_bounds}
There exists a vector \(\tilde{\ell}\in\mathbb R^N_{\ge 0}\) such that $|r_i(t)|=|(H\tilde{q}(t)+\delta(t))|_i\le \tilde{\ell}_i, \forall t\ge 0,\ i\in\mathcal N$.
 \hfill $\square$
\end{assumption}

\begin{assumption} \label{assum:spectral_radius}
Let $D:=\operatorname{diag}(d_1,\ldots,d_N)$ and $A=D|F|$, where \(|F|\) denotes the elementwise absolute value of \(F\). Then, it holds that $\varrho(A)<1$, where
\(\varrho(\cdot)\) denotes the spectral radius. \hfill $\square$
\end{assumption}

The following result then establishes the forward invariance of the virtual droop states. We remind the reader that $|z(0)|$ denotes the elementwise vector of absolute values, i.e., $|z(0)|=(|z_i(0)|)_{i \in \mathcal{N}}$.

\begin{lemma} \label{lem:invariance}
Consider the virtual droop dynamics
in \eqref{virtual_controller} and the output dynamics in \eqref{new_output}. Suppose that Assumption \ref{assum:sector-bounded}, \ref{assum:res_bounds} and \ref{assum:spectral_radius} are satisfied.  Let \(r_z\in\mathbb R^N_{\ge0}\) be a decision variable $r_{z,i}\ge |z_i(0)|$ and $r_{z,i}
\ge
d_i\left(
\sum\limits_{j \in \mathcal{N}} |F_{ij}|r_{z,j}
+\tilde{\ell}_i
\right)$,
for any $i\in\mathcal N$. Then, the following statements hold: 
\begin{enumerate}
\item An admissible finite choice of \(r_z\in\mathbb R^N_{\ge0}\) is
$r_z
=
(I-D|F|)^{-1}
\left(
|z(0)|+D\tilde{\ell}
\right)$.
\item 
The set $\mathcal Z
:=
\left\{
z\in\mathbb R^N:\ |z_i|\le r_{z,i},\ i\in\mathcal N
\right\}$
is forward invariant, i.e., 
$|z_i(t)|\le r_{z,i},  \forall t\ge0,\ $ and any $i\in\mathcal N$. \ \ \ \ \ \ \ \ \hfill $\square$
\end{enumerate}
\end{lemma}

\emph{Proof}: See Appendix. \hfill $\blacksquare$

Lemma \ref{lem:invariance} can be leveraged to provide certificates for the
closed-loop system voltages as follows.

\begin{theorem} \label{thm:preconditioned_voltage_certificates}
Consider the virtual droop dynamics
in \eqref{virtual_controller} and the output dynamics in \eqref{new_output}. Suppose that Assumptions \ref{assum:sector-bounded}, \ref{assum:res_bounds} and \ref{assum:spectral_radius} are satisfied. Furthermore, consider the setting of Lemma \ref{lem:invariance}. It then holds that for each $i \in \mathcal{N}$ and for all $t \geq 0$ $|y_i(t)|
\le
\tilde{b}_i$, where $\tilde{b}_i
    :=
    \sum\limits_{j=1}^N |(HG)_{ij}|r_{z,j}
    +
    \tilde{\ell}_i$. \hfill $\square$

\end{theorem}

\emph{Proof:} 
By Lemma \ref{lem:invariance}, the set 
\begin{align}
\mathcal Z
=
\left\{
z\in\mathbb R^N:\ |z_i|\le r_{z,i},\ i\in\mathcal N
\right\} \nonumber 
\end{align}
is forward invariant, i.e., for all \(t\ge 0\) and all
\(j\in\mathcal N\), $|z_j(t)|\le r_{z,j}$.
Using \eqref{new_output} and applying the triangle inequality
we have, for each $i\in\mathcal{N}$, that it holds
$|y_i(t)|
\le
\sum_{j=1}^N |F_{ij}|\,|z_j(t)|+|r_i(t)|$.
Using \(|z_j(t)|\le r_{z,j}\) and
Assumption \ref{assum:res_bounds}, we obtain
\[
|y_i(t)|
\le
\sum_{j \in \mathcal{N}} |(HG)_{ij}|r_{z,j}+\tilde{\ell}_i,
\]
where $F=HG$.
This concludes the proof. \hfill 
\(\blacksquare\) \\

Theorem~\ref{thm:preconditioned_voltage_certificates} applies to a general
coordination matrix \(G\). The following corollary specializes the
result to a coordination matrix that diagonalizes the effective voltage
sensitivity.

\begin{corollary} \label{cor:simplified}
\label{cor:inverse-sensitivity-certificate}
Let the conditions of
Theorem~\ref{thm:preconditioned_voltage_certificates} and
Lemma~\ref{lem:invariance} hold. Furthermore, let
$G=-\alpha H^{-1}, \alpha>0$, and assume that $\alpha d_i<1, i\in\mathcal N$.
Then an admissible exact choice of the invariant state bound for each $i \in \mathcal{N}$ is $
r_{z,i}
=
\max\left\{
|z_i(0)|,\,
\frac{d_i\widetilde\ell_i}{1-\alpha d_i}
\right\}$, and the corresponding voltage certificate is simplified to
\[
|y_i(t)|
\leq
\widetilde b_i
:=
\alpha
\max\left\{
|z_i(0)|,\,
\frac{d_i\widetilde\ell_i}{1-\alpha d_i}
\right\}
+
\widetilde\ell_i.
\]
If, additionally,  \(z_i(0)=0\), \(d_i=d\), and
\(\widetilde\ell_i=\bar r\) for every \(i\in\mathcal N\), then
\(\widetilde b_i=\dfrac{\bar r}{1-\alpha d}\). \hfill $\square$
\end{corollary}
\emph{Proof:}
The result follows directly from \(HG=-\alpha I\), together with
Lemma~\ref{lem:invariance} and
Theorem~\ref{thm:preconditioned_voltage_certificates}.  \hfill $\blacksquare$

Corollary~\ref{cor:inverse-sensitivity-certificate} shows that
\(G=-\alpha H^{-1}\) removes, if the inverter limits allow it, the feeder-induced downstream
deterioration from the certificate.
The resulting bound is then governed by the residual disturbance $\bar{r}$ and the
margin \(1-\alpha d\).
 
\subsection{Feedforward Design for  Inverter Limits}
Lemma 3 and Theorem 2 provide bounds on $z(t)$ and $y(t)$,
but do not ensure that the command
$q_g(t)=Gz(t)+\tilde q(t)$ satisfies the inverter limits.
We therefore consider the affine feedforward law
\begin{equation}
\tilde{q}(t)=q_\phi+y_0(t)q_y,\label{eq:feedforward}
\end{equation}
where \(q_\phi,q_y\in\mathbb R^N\) are design variables.
For fixed $G$, the following problem (P) jointly bounds the residual
disturbance and enforces the inverter reserve constraints:
\[
\label{eq:reserve_aware_feedforward_lp}
\left\{
\begin{alignedat}{2}
\min_{q_\phi,q_y,\mu,r_z}\quad
& \mu+\lambda\sum_{i=1}^N r_{z,i} \\[1mm]
\mathrm{s.t.}\quad
& -\mu
\le
\phi'_i+y'_0
+
\sum_{j \in \mathcal{N}} H_{ij}
\left(q_{\phi,j}+y'_0 q_{y,j}\right)
\le
\mu, \\[1mm]
&\sigma\bigl(q_{\phi,i}+y_0' q_{y,i}\bigr)
 +\sum_{j\in\mathcal N}|G_{ij}|r_{z,j}
 \leq Q_i^{\max},
\\[-1mm]
&\hspace{15mm}
\sigma\in\{-1,1\}. \\
&
r_{z,i}\ge |z_i(0)|, \\[1mm]
&
r_{z,i}
\ge
d_i
\left(
\sum_{j \in \mathcal{N}} |(HG)_{ij}|r_{z,j}
+
\mu
\right), \\[1mm]
&
\mu\ge 0,\quad r_{z,i}\ge 0, \\[1mm]
&
\forall \phi'_i\in\{-\bar{\phi}_i,\bar{\phi}_i\},
\quad
y'_0\in \{\underline{y}_0, \overline{y}_0\}, \ 
\forall i\in\mathcal N, 
\end{alignedat}
\right.
\]
 where
\(\lambda\ge0\) is an appropriately chosen weight and $\bar{\phi_i}$ is the bound obtained for $|\phi_i(t)|$ in Lemma 1. Performance and safety certificates are then obtained in Theorem \ref{thm:opt}. \\

\begin{theorem} \label{thm:opt}
Fix $G$, 
\(z(0)\), and \(d_i\ge 0\), for all \(i\in\mathcal N\). Under Assumptions 1, 2, 3 and 5,
 consider the feedforward controller in \eqref{eq:feedforward}. If problem (P) is feasible, then,
for any feasible  \((q_\phi,q_y,\mu,r_z)\), the controller $q_g(t)=Gz(t)+\tilde{q}(t)$ satisfies that:
 \begin{enumerate}[label=\roman*)]
 \item $|r_i(t)|\le \mu$ and  $|z_i(t)|\le r_{z,i}$, 
 \item $|y_i(t)|
    \le
    \sum\limits_{j \in \mathcal{N}} |(HG)_{ij}|r_{z,j}+\mu$ and \item $|q_{g,i}(t)|\le Q_i^{\max}$,
    \end{enumerate}
for all $ t\ge 0$ and for each $i \in \mathcal{N}$ . \hfill $\square$
\end{theorem}
\emph{Proof:} i) Let \((q_\phi,q_y,\mu,r_z)\) be any feasible solution of problem
(P). Note that the residual $r_i(t)=\phi_i(t)+ y_0(t)+\sum_{j \in \mathcal{N}}H_{ij}(q_{\phi,j}+ y_0 q_{y, j})$ is an affine function of $y_0$ and $\phi_i(t)$. By Lemma 1, we have that $\phi_i(t)\in [-\bar{\phi_i}, \bar{\phi_i}]$ and by Assumption \ref{assum:feeder-head-voltage}, $y_0(t) \in [\underline{y}_0, \overline{y}_0 ]$. As such, feasibility for the constraint $|r_i(t)| \leq \mu$ for all $t \geq 0$ is ensured if the following constraints are satisfied:
\begin{align}
\left| \phi'+y'_0
    +
    \sum_{j \in \mathcal{N}} H_{ij}
    \left(
        q_{\phi,j}+y'_0 q_{y,j}
    \right) \right|
    \le
    \mu,
\end{align}
for all $\phi' \in \{-\bar{\phi}_i, \bar{\phi}_i\}$ and $y'_0 \in \{\underline{y}_0, \overline{y}_0\}$.
Consider now Assumption \ref{assum:spectral_radius} and the decision variable $r_{z,i}$ for which the constraints of Lemma \ref{lem:invariance} are satisfied with $\tilde{\ell}_i=\mu$ for each agent $i \in \mathcal{N}$.
Then, from Lemma \ref{lem:invariance}  the set $\mathcal{Z}$ is forward invariant, i.e., $|z_i(t)| \leq r_{z,i} \forall t \geq 0, i \in \mathcal{N}$. 

ii) As such, the bound on $|y_i(t)|$ is provided as follows:
\begin{align}
|y_i(t)| = &\left|\sum_{j \in \mathcal{N}}(HG)_{ij}z_j(t)+r_i(t)\right| \nonumber \\
\leq &  \sum_{j \in \mathcal{N}}|(HG)_{ij}||z_j(t)|+|r_i(t)| \nonumber \\
\leq &  \sum_{j \in \mathcal{N}}|(HG)_{ij}|r_{z,j}+\mu.
\end{align}
iii) It remains to show that the physical inverter constraints are satisfied.
For each $i\in\mathcal N$, due to \eqref{virtual_controller} and \eqref{eq:feedforward}, and using $|z_j(t)|\le r_{z,j}$ from $i)$, it follows that
\[
|q_{g,i}(t)|
\le
\sum_{j=1}^N |G_{ij}|r_{z,j}
+
|q_{\varphi,i}+y_0(t)q_{y,i}|.
\]
The two reserve constraints in $(P)$, imposed for both vertices
$y_0'\in Y_0$, give
\begin{equation}
|q_{\varphi,i}+y_0' q_{y,i}|
+
\sum_{j=1}^N |G_{ij}|r_{z,j}
\le
Q_i^{\max}. \label{eq:Q_max_con}
\end{equation}
Since the function in absolute values is convex in $y_0'$, the same inequalities
hold for every $y_0(t)\in[\underline y_0,\overline y_0]$. Hence, \eqref{eq:Q_max_con} ensures that
\[
\sum_{j=1}^N |G_{ij}|r_{z,j}
+
|q_{\varphi,i}+y_0(t)q_{y,i}|
\le
Q_i^{\max},
\]
and thus
$|q_{g,i}(t)|\le Q_i^{\max}, \forall t\ge 0,\ i\in\mathcal N$. \hfill $\blacksquare$

\Autom{Problem~$(\textrm{P})$ shows how the voltage certificates can be
used to shape the closed-loop response for a fixed coordination
matrix $G$, while enforcing the inverter limits. In the ideal case
where the disturbance $\delta(t)$ is known, Lemma~2 implies that it
can be cancelled exactly by choosing
$\tilde q(t)=-H^{-1}\delta(t)$. For bounded but unknown disturbances,
the affine feedforward law instead minimizes their residual effect
through $\mu$. }
\begin{comment}
Problem $\mathrm{(P)}$ serves three main roles: 1) Its solutions produce a residual bound \(\mu\) replacing the bound \(\tilde\ell_i\) in Theorem~2. 2) It guarantees forward invariance of the virtual state set $\mathcal{Z}$ and through this property provides bounds for each $y_i$. 3) It enforces the physical reserve constraints \(|q_{g,i}(t)|\le Q_i^{\max}\). Thus, $\mathrm{(P)}$ ensures that the voltage certificate is compatible with the inverter limits. Since \(H\) is nonsingular by Lemma~2, exact cancellation is possible when \(\delta(t)\) is known. Specifically, choosing $\tilde{q}(t)=-H^{-1}\delta(t)$ gives \(r(t)=H\tilde{q}(t)+\delta(t)=0\). Considering the controller in \eqref{eq:feedforward}, this is obtained with $q_\phi=-H^{-1}\phi, q_y=-H^{-1}\mathbf 1$ . This exact cancellation choice is a constant feedforward law only when \(\phi(t)\) is known and constant; otherwise it would require an online update \(q_\phi(t)=-H^{-1}\phi(t)\). For bounded but unknown \(\phi(t)\), Problem~ (P) instead minimizes the worst-case residual \(\mu\).
\end{comment}

% Requires:
% \usepackage{tikz}
% \usetikzlibrary{arrows.meta,positioning,fit,calc}

% Requires:
% \usepackage{tikz}
% \usetikzlibrary{arrows.meta,positioning,fit,calc}

% ============================================================
% Sparse column
% ============================================================

\section{Joint Synthesis of Coordinated Droop Control and Voltage Certificates}

Previously, we considered a coordinated droop controller with a fixed coordination matrix $G$ and a feedforward controller whose parameters can be obtained through an optimization problem. In this section, we depart from fixing $G$ a priori and instead treat it as a controller synthesis variable subject to a prescribed communication mask. The objective of this section is to jointly construct the coordinated controller and voltage certificates under this mask, while guaranteeing forward invariance and feasibility of the physical inverter limits. However, direct optimization over \(G\) and the invariant set radii \(r_z\) would lead to bilinear terms such as \(|G_{ij}|r_{z,j}\) in the optimization procedure. To avoid this, we
introduce an auxiliary design matrix $
U:=GR_z$,
where 
$R_z:=\operatorname{diag}(r_{z,1},\ldots,r_{z,N})$. We impose the following assumption: \\
\begin{assumption} \label{assum:r_z_pos}
For each bus $i \in \mathcal{N}$, $r_{z,i}$ is a strictly positive scalar. \hfill $\square$  
\end{assumption}
Under Assumption \ref{assum:r_z_pos}, $R^{-1}_z$ exists and $G$ can be retrieved as $G=UR_z^{-1}$.
This change of variables allows the feedback contributions from the neighbourhood of controllers as well as the  inverter reserve constraints to be affine in the decision
variables. Let us now define \(\mathcal M\subseteq\mathcal N\times\mathcal N\) to be
a prescribed communication mask imposing restrictions on $G$. Specifically, the constraint $G_{ij}=0,\ (i,j)\notin\mathcal M$,
means that inverter \(i\) is not allowed to use the virtual state \(z_j\).
 We consider the controller used in \eqref{virtual_controller}
 with the same choice of feedforward control in \eqref{eq:feedforward}.  The following preliminary result shows that the communication restrictions imposed on \(U\) are
inherited by \(G\).

\begin{lemma} \label{lem:scaled_sparse_representation}
Consider \Autom{the matrix} $U:=GR_z$, where \\ $R_z:=\operatorname{diag}(r_{z,1},\ldots,r_{z,N})$. Under Assumption \ref{assum:r_z_pos},
if $U_{ij}=0$, \Autom{for all} $(i,j)\notin\mathcal M$,
then $G_{ij}=0$, \Autom{for all} $(i,j)\notin\mathcal M$. Moreover, if \(z\in Z(r_z):=\{z\in\mathbb R^N: |z_i|\le r_{z,i}\}\), then there exists \Autom{a vector} \(\xi\in\mathbb R^N\) such that
$z=R_z\xi$ \Autom{and} $Gz=U\xi$ with $|\xi_j|\le 1$ . \hfill $\square$
\end{lemma}

\emph{Proof}: Since \(R_z^{-1}\) is diagonal, right multiplication by \(R_z^{-1}\)
only rescales the elements of \(U\). Hence the zero pattern of \(U\) is
preserved under the transformation \(G=UR_z^{-1}\). If \(|z_i|\le r_{z,i}\) and $\xi_i:=\dfrac{z_i}{r_{z,i}}$, then \(|\xi_i|\le 1\). Then, it holds that \(z=R_z\xi\) and $Gz=UR_z^{-1}R_z\xi=U\xi$. \hfill $\blacksquare$ 

We denote the scaled voltage sensitivity by $W:=HU$.
The matrix \(W\) describes the voltage sensitivity with respect to the
normalized virtual state \(\xi\), since \(z=R_z\xi\) and \(Gz=U\xi\).
To avoid the degenerate solution \(U=0\), and to ensure that each local
virtual droop state participates in the corresponding voltage feedback
channel, we impose a diagonal shaping condition. To this end, let
\(0<\underline\alpha_i\leq \bar\alpha_i\) be prescribed constants.\Autom{We then impose the condition:}
\[
        -\bar\alpha_i r_{z,i}
        \leq W_{ii}
        \leq
        -\underline\alpha_i r_{z,i},  i\in\mathcal N, 
\] where $\underline{\alpha}_i, \overline{\alpha}_i$ are positive scalars. Note that these constraints do not force the controller \(G\) to be diagonal.
They only impose nonzero negative diagonal elements in the shaped voltage
matrix \(W=HU\), with magnitude proportional to the invariant radius
\(r_{z,i}\). Off-diagonal entries of \(W\), and therefore nonlocal
coordination through \(G\) are still allowed.

\begin{lemma} \label{lem:neighborhood_voltage_envelope}
Suppose that Assumptions 1, \ref{assum:feeder-head-voltage}, and 6 hold. Consider the virtual controller in (\ref{virtual_controller}), with the affine feedforward law in (\ref{eq:feedforward}), and let \(U=GR_z\),  where $R_z:=\operatorname{diag}(r_{z,1},\ldots,r_{z,N})$. Suppose that \(z(t)\in Z(r_z)\) for all $t \geq 0$ and let \(A_{ij}\ge 0\) and \(\mu_i\ge 0\) satisfy 
\begin{equation} -A_{ij}\le W_{ij}\le A_{ij}, \text{ for all } i,j\in\mathcal N, \label{eq:HU-envelope} \end{equation}
and 
\begin{equation} 
\left| \phi_i' + y_0' + \sum_{k\in\mathcal N} H_{ik} \left(q_{\phi,k}+y_0' q_{y,k}\right) \right| \le \mu_i, \label{eq:residual-envelope} 
\end{equation}
for all $\phi_i'\in\{-\bar\phi_i,\bar\phi_i\}$, and $ y_0'\in Y_0 .$ Then, for every admissible \(\phi_i(t)\in[-\bar\phi_i,\bar\phi_i]\) and every \(y_0(t)\in Y_0\), it holds that
 \begin{equation} |y_i(t)| \le \sum_{j\in\mathcal N} A_{ij}+\mu_i, \label{eq:voltage-envelope-raw} \nonumber 
\end{equation}
for all $t \geq 0$ and for all $i\in\mathcal N$. \hfill $\square$
\end{lemma}

\emph{Proof:} By Lemma 4, since \(z(t)\in Z(r_z)\), there exists \(\xi(t)\) with \(|\xi_j(t)|\le 1\) such that \(Gz(t)=U\xi(t)\). Substituting the controller into the voltage equation gives \[ y(t) = W\xi(t) + H\left(q_\phi+y_0(t)q_y\right) + \phi(t) + y_0(t)\mathbf 1 . \] 
For each \(i\in\mathcal N\) it holds that  $y_i(t) = \sum_{j\in\mathcal N}W_{ij}\xi_j(t) + r_i(t)$. From \eqref{eq:HU-envelope} and \(|\xi_j(t)|\le 1\), we obtain \[ \left| \sum_{j\in\mathcal N}(HU)_{ij}\xi_j(t) \right| \le \sum_{j\in\mathcal N}A_{ij}. \] Moreover, \(r_i(t)\) is affine in \(\phi_i(t)\) and \(y_0(t)\). By Lemma 1 and Assumption \ref{assum:feeder-head-voltage}, the admissible set is $\phi_i(t)\in[-\bar\phi_i,\bar\phi_i]$ \Autom{and} $y_0(t)\in Y_0 . $ Since the maximum absolute value of an affine function over this rectangle is attained at a vertex, the vertex constraints \eqref{eq:residual-envelope} imply $|r_i(t)|\le \mu_i.$ Hence,  by application of the triangle inequality  it holds that:  \[ |y_i(t)| \le \left| \sum_{j\in\mathcal N}(HU)_{ij}\xi_j(t) \right| + |r_i(t)| \le \sum_{j\in\mathcal N}A_{ij}+\mu_i, \] which concludes the proof. \hfill $\blacksquare$

\subsection{Stability and safety}

\Autom{To turn the voltage envelope of Lemma~5 into an all-time
certificate, the virtual state must remain in $\mathcal Z(r_z)$,
while the reactive power input must respect the inverter limits.
Since the slope restriction of $K_i$ implies
$|K_i(y_i)|\leq d_i b_i$ whenever $|y_i|\leq b_i$,
choosing $r_{z,i}\geq d_i b_i$ makes the vector field
inward-pointing on the boundary of $\mathcal Z(r_z)$.}
\begin{comment}
The voltage envelope from Lemma~ \ref{lem:neighborhood_voltage_envelope} is useful only if the virtual state
remains inside the set \(Z(r_z)\). As such, deriving conditions to ensure that this is the case is of utmost importance. Furthermore, additional constraints are required to ensure that the physical reactive power input remains within the inverter
limits. Note that the slope restriction of \(K_i\) implies that whenever \(|y_i|\le b_i\),
one has \(|K_i(y_i)|\le d_i b_i\). Thus, choosing
\(r_{z,i}\ge d_i b_i\) ensures that the vector field points inward on the
boundary of \(Z(r_z)\). This intuitive observation is formalized in the following result.
\end{comment}
\begin{lemma} 
\label{lem:neighborhood_forward_invariance} 
 Suppose that Assumptions 1, \ref{assum:feeder-head-voltage}, and 6 hold and consider the setting of Lemma \ref{lem:neighborhood_voltage_envelope}. Suppose that, for each \(i\in\mathcal N\), it holds for some chosen $b_i$ that
    $\sum_{j\in\mathcal N} A_{ij}+\mu_i\le b_i$,
and
    $r_{z,i}\ge |z_i(0)|$,
    $r_{z,i}\ge d_i b_i$,
    $r_{z,i}\ge \underline r_i>0$.
Then \(Z(r_z)\) is forward invariant. \hfill $\square$
\end{lemma}
\emph{Proof:}   By Lemma~5, for every \(z \in Z(r_z)\), it holds that $|y_i(t)| \le
\sum_{j \in \mathcal{N}} A_{ij}
+
\mu_i$.
Under Assumption \ref{assum:sector-bounded}, we have
\begin{align}
|K_i(y_i)|
&=
|K_i(y_i)-K_i(0)|
\le d_i |y_i| \nonumber \\
& \le d_i
\left(
\sum_{j \in \mathcal{N}} A_{ij}
+
\mu_i
\right)
\le r_{z,i}. \nonumber 
\end{align}
Hence, on the boundary \(z_i=r_{z,i}\), it holds that $\dot z_i
=
-\tau_i^{-1} r_{z,i}
+
\tau_i^{-1} K_i(y_i)
\le 0$,
while on the boundary \(z_i=-r_{z,i}\), one has
$\dot z_i
=
\tau_i^{-1} r_{z,i}
+
\tau_i^{-1} K_i(y_i)
\ge 0$.
Thus, the vector field points inward on every boundary face of \(Z(r_z)\).
Since \(r_{z,i}\ge |z_i(0)|\), we have that \(z(0)\in Z(r_z)\). Therefore,
\(Z(r_z)\) is forward invariant.  \hfill $\blacksquare$ \\
The next
result provides linear reserve constraints that guarantee this physical
quantity remains within \(Q_i^{\max}\).
\begin{proposition} 
    \label{prop:neighborhood_inverter_limits}
Suppose that there exist \(V_{ij}\ge0\), with \(i,j\in\mathcal N\), such that the following inequalities hold:
\begin{align}
    -V_{ij}\le U_{ij}\le V_{ij},\ V_{ij}\ge0, 
\label{eq:U_bound_constraints}
\end{align}
and, for all \(y_0'\in Y_0\), it holds that:
\begin{align}
 q_{\varphi,i}+y_0' q_{y,i}
+
\sum_{j=1}^N V_{ij}
\le Q_i^{\max}, \forall i \in \mathcal{N} \label{eq:reserve_constraints1}
\end{align}
\begin{align}
-(q_{\varphi,i}+y_0' q_{y,i})
+
\sum_{j=1}^N V_{ij}
\le Q_i^{\max},  \label{eq:reserve_constraints2} \forall i \in \mathcal{N}.
\end{align}
Then, for every trajectory satisfying \(z(t)\in Z(r_z)\), it holds that $|q_{g,i}(t)|\le Q_i^{\max}, \Autom{\text{ for all }} t\ge 0$. \hfill $\square$

\end{proposition}

\emph{Proof:}  \Autom{For any} \(z(t)\in Z(r_z)\), by Lemma~4 there exists
\(\xi\in\mathbb{R}^N\) such that \(|\xi_j|\le 1\) and $Gz(t)=U\xi$ . Thus, for each \(i\in\mathcal{N}\) it holds that 
\begin{align}
|(Gz(t))_i|
=
\left|\sum_{j=1}^N U_{ij}\xi_j\right|
\le
\sum_{j=1}^N |U_{ij}|
\le
\sum_{j=1}^N V_{ij}. \nonumber 
\end{align}
Moreover, the feedforward term $\tilde q_i(t)=q_{\phi,i}+y_0(t)q_{y,i}$ is affine in \(y_0(t)\), and
\(y_0(t)\in[\underline y_0,\overline y_0]\), which implies that:
\[
|\tilde q_i(t)|+\sum_{j=1}^N V_{ij}
\le Q_i^{\max},
\]
\Autom{for all $t\ge 0$. This provides a sufficient condition as follows:}
\[
|q_{g,i}(t)|
=
|(Gz(t))_i+\tilde q_i(t)|
\le
|(Gz(t))_i|+|\tilde q_i(t)|
\le Q_i^{\max}, 
\]
which concludes the proof. \hfill $\blacksquare$

\subsection{Controller Synthesis}
In Section~4,
the spectral radius condition in Assumption~\ref{assum:spectral_radius} was used to establish the
existence of an invariant box for a fixed coordination matrix. In the
synthesis problem below, this condition is not imposed directly. Instead,
we enforce forward invariance through the linear sufficient condition
\(r_{z,i}\geq d_i b_i\), together with the voltage envelope constraint
\(|y_i(t)|\leq b_i\). This guarantees that
\(|K_i(y_i(t))|\leq r_{z,i}\), and hence that the vector field points
inward on the boundary of the chosen box.
Based on the results of the previous section, we  propose the following optimization problem ($\mathrm{\tilde{P}}$):

\begin{equation}
\label{eq:pcvc-lp}
\left\{
\begin{aligned}
&
\min_{\substack{U,A,V,\mu,r_z,b,\\ q_\phi,q_y}}
\quad
\|b\|_1
+\varepsilon_r\|r_z\|_1
+\varepsilon_A\|A\|_{1,1}
+\varepsilon_V\|V\|_{1,1}
\\[-0.5mm]
&
\mathrm{s.t.} \ \ \  \quad U_{ij}=0, (i,j)\notin\mathcal M,
\\
&
\qquad -A\leq W\leq A, \nonumber \\
&\qquad -V\leq U\leq V,
\\
&
\qquad -\mu+\bar\phi
\leq
y_0'\mathbf 1+H(q_\phi+y_0'q_y)
\leq
\mu-\bar\phi,
y_0'\in\mathcal Y_0^{\mathrm v},
\\
&
\ \ \ \ \qquad A\mathbf 1+\mu\leq b,
\\
&
\qquad \ \ \  \ r_z\geq |z(0)|, \  r_z\geq Db, \ r_z\geq\underline r,
\\
&
\qquad -Q^{\max}+V\mathbf 1
\leq
q_\phi+y_0'q_y
\leq
Q^{\max}-V\mathbf 1,
y_0'\in\mathcal Y_0^{\mathrm v},
\\
&
\qquad -\overline\alpha\odot r_z
\leq
\operatorname{diag}(W)
\leq
-\underline\alpha\odot r_z.
\end{aligned}
\right.
\end{equation}
where  $W:=HU$,
$D:=\operatorname{diag}(d_1,\ldots,d_N)$, $\bar\phi
:=(\bar\phi_1,\ldots,\bar\phi_N)^\top$,
$Q^{\max}
:=(Q_1^{\max},\ldots,Q_N^{\max})^\top$,
$\mathcal Y_0^{\mathrm v}:=\{\underline y_0,\overline y_0\}$,
$\underline\alpha
:=
(\underline\alpha_1,\ldots,\underline\alpha_N)^\top$, and
$\overline\alpha
:=
(\overline\alpha_1,\ldots,\overline\alpha_N)^\top$ \Autom{have been written in compact form}.
All vector and matrix inequalities below are understood componentwise,
\Autom{while} \(\odot\) denotes componentwise multiplication.  Here, $\underline{r}\in\mathbb{R}_{>0}^{N}$ and
$\epsilon_r,\epsilon_A,\epsilon_V\geq 0$ are chosen constants, while
$A,V\in\mathbb{R}_{\geq 0}^{N\times N}$,
$\mu,b,r_z\in\mathbb{R}_{\geq 0}^{N}$, and
$\|A\|_{1,1}:=\sum_{i,j}|A_{ij}|$, with $\|V\|_{1,1}$
defined analogously.

The primary objective in \((\widetilde{\mathrm{P}})\) is to minimize the voltage certificate
\(b\). The remaining terms are regularizers, where 
\(\|r_z\|_1\) discourages large invariant boxes,
\(\|A\|_{1,1}\) discourages large voltage sensitivities, and
\(\|V\|_{1,1}\) discourages excessive use of \Autom{reactive power reserve}. Thus, every
feasible solution of \((\widetilde{\mathrm{P}})\) defines a safe controller, while an optimal
solution selects, among the feasible safe controllers, one that minimizes
the chosen certificate objective. 
Owing to its separable objective and linear coupling constraints,
the synthesis problem is amenable to distributed optimization methods.
\begin{theorem}
\label{thm:neighborhood_sparse_synthesis}
For any feasible solution of \((\widetilde{\mathrm{P}})\), define
\(R_z^\ast:=\operatorname{diag}(r_{z,1}^\ast,\ldots,r_{z,N}^\ast)\)
and \(G^\ast:=U^\ast(R_z^\ast)^{-1}\). Then the controller $ q_g(t)=G^\ast z(t)+q_\phi^\ast+y_0(t)q_y^\ast$
satisfies, for all \(t\ge 0\) and all \(i\in\mathcal N\), the conditions $|z_i(t)|\le r^\ast_{z,i}$, $|q_{g,i}(t)|\le Q_i^{\max}$,
and $|y_i(t)|\le b^\ast_i$. \hfill $\square$
\end{theorem}

\emph{Proof:}
Since \(r_{z,i}^\ast\ge \underline r_i>0\), the matrix \(R_z^\ast\)
is invertible. The equivalence in sparsity between $U^\ast$ and \(G^\ast\) follows from
Lemma~\ref{lem:scaled_sparse_representation}. 
Lemma~\ref{lem:neighborhood_voltage_envelope} then implies that if
\(z(t)\in Z(r_z^\ast)\) then it holds that:
\[
    |y_i(t)|
    \le
    \sum_{j\in\mathcal N}A_{ij}^\ast+\mu_i^\ast
    \le b_i^\ast .
\]
From Lemma~\ref{lem:neighborhood_forward_invariance}, sufficient conditions for the condition \(z(t)\in Z(r_z^\ast)\)  to hold are  
$r_{z,i}^\ast\ge |z_i(0)|$,
    $r_{z,i}^\ast\ge d_i b_i^\ast$, and 
    $r_{z,i}^\ast\ge \underline r_i>0$.
 As such, it holds that 
    $|z_i(t)|\le r_{z,i}^\ast$, \Autom{for all $t\ge0$, and all $i\in\mathcal N$}.
Since \(z(t)\in Z(r_z^\ast)\) for all \(t\ge0\), the voltage bound above
also holds for all \(t\ge0\). Finally, constraints
\eqref{eq:U_bound_constraints}, \eqref{eq:reserve_constraints1}, and
\eqref{eq:reserve_constraints2} imply, by
Proposition~\ref{prop:neighborhood_inverter_limits}, that $|q_{g,i}(t)|\le Q_i^{\max}$,
    \Autom{for all $t\ge0$, and each $i\in\mathcal N$} .
This concludes the proof.  \hfill $\blacksquare$ 

Theorem~\ref{thm:neighborhood_sparse_synthesis} states that any feasible solution of ($\tilde{\mathrm P}$) yields an implementable coordinated controller that respects the operational constraints while certifying all bus voltages under admissible disturbances. Thus, program ($\tilde{\mathrm P}$) jointly determines the controller and its respective  certificates.
\subsection{Sparse inverter topology}
The developments above assume that every monitored bus is equipped with a
controllable inverter. 
In this section, we show how our setting can be extended to networks with sparse inverter placement. To this end, let
\(\mathcal I=\{\iota_1,\ldots,\iota_m\}\subset \mathcal N\) denote the subset of
buses equipped with controllable inverters, with \(m\leq N\), and 
\(S_{\mathcal I}\in\{0,1\}^{m\times N}\) be the selection matrix that extracts
the components associated with inverter buses, so that
\(x_{\mathcal I}=S_{\mathcal I}x\) for any \(x\in\mathbb R^N\). In a similar fashion, a
vector of controllable reactive power injections
\(q_{\mathcal I}\in\mathbb R^m\) can be embedded into the full corresponding vector by  the transformation $q_g = S_{\mathcal I}^{\top} q_{\mathcal I}$ such that  \(q_{g,i}=0\) for all \(i\notin\mathcal I\), i.e.,  buses in $\mathcal{N} \setminus \mathcal{I}$ have no controllable reactive power injection. Substituting into \eqref{output_dynamics} it holds that:
\[
        y_{\mathcal N}
        = H_{\mathcal N\mathcal I}q_{\mathcal I}
          + \phi_{\mathcal N} + y_0 1_N,
\]
where $H_{\mathcal N\mathcal I}:=H S_{\mathcal I}^{\top}
        \in\mathbb R^{N\times m}.$ The feedback used by the droop controller is evaluated only at buses that have controllable reactive
power with output
\begin{align}
        y_{\mathcal I}
        = S_{\mathcal I} y_{\mathcal N}
        &= S_{\mathcal I}H S_{\mathcal I}^{\top}q_{\mathcal I}
          + S_{\mathcal I}\phi_{\mathcal N} + y_0 1_m \nonumber \\
        &= H_{\mathcal I\mathcal I}q_{\mathcal I}
          + \phi_{\mathcal I} + y_0 1_m, \nonumber 
\end{align}
where $H_{\mathcal I\mathcal I}:=
        S_{\mathcal I}H S_{\mathcal I}^{\top}\in\mathbb R^{m\times m}$ and 
        $\phi_{\mathcal I}:=S_{\mathcal I}\phi_{\mathcal N}.$
 The corresponding inverter dynamics then take the form 
 \begin{align}
 \dot q_{\mathcal I}
        =
        -T_{\mathcal I}^{-1}q_{\mathcal I}
        +T_{\mathcal I}^{-1}K_{\mathcal I}(y_{\mathcal I}), \nonumber  
      \end{align}  
        where  $T_{\mathcal I}:=\mathrm{diag}(\tau_{\iota_1},\ldots,\tau_{\iota_m})$, and 
        $K_{\mathcal I}(y_{\mathcal I})
        :=
        (K_{\iota_1}(y_{\iota_1}),\ldots,K_{\iota_m}(y_{\iota_m}))^\top$ .
The virtual controller with sparse inverter placement is then written as
\begin{equation}
\left\{
\begin{aligned}
q_{\mathcal I}(t)
&=
G_{\mathcal I}z_{\mathcal I}(t)
+q_\phi+y_0(t)q_y,\\
\dot z_{\mathcal I}(t)
&=
-T_{\mathcal I}^{-1}z_{\mathcal I}(t)
+T_{\mathcal I}^{-1}K_{\mathcal I}(y_{\mathcal I}(t)).
\end{aligned}
\right.
\end{equation}
where \(G_{\mathcal I}\in\mathbb R^{m\times m}\) and
\(q_\phi,q_y\in\mathbb R^m\).  Defining the matrices
\(R_z=\mathrm{diag}(r_{z,1},\ldots,r_{z,m})\) and
\(U=G_{\mathcal I}R_z\) for notational convenience, any communication restriction is now imposed on \(G_{\mathcal I}\), or equivalently on \(U\). 
\begin{comment}
We note that in the sparse setting,
\(z_{\mathcal I},q_{\mathcal I},q_\phi,q_y,r_z\in\mathbb R^m\) and
\(G_{\mathcal I},U,V,R_z\in\mathbb R^{m\times m}\), while
\(y_{\mathcal N},b_{\mathcal N},\mu\in\mathbb R^N\) and
\(A\in\mathbb R_{\geq0}^{N\times m}\).
\end{comment}
The output equation then takes the form: 
\begin{equation}
        y_{\mathcal N}(t)
        =
        H_{\mathcal N\mathcal I}U\xi(t)
        +H_{\mathcal N\mathcal I}
          \big(q_\phi+y_0(t)q_y\big)
        +\phi_{\mathcal N}(t)+y_0(t)1_N, \nonumber 
\end{equation}

where \(|\xi_j(t)|\leq 1\) whenever
\(|z_{\mathcal I,j}(t)|\leq r_{z,j}\). Thus, the case with sparse inverters has the
same structure as the fully actuated case, but with two distinct sensitivity
matrices: i) \(H_{\mathcal N\mathcal I}\) is instead used to certify all
voltages, and  ii) \(H_{\mathcal I\mathcal I}\) which is used for the feedback
constraints at inverter buses. The synthesis problem is thus directly obtained from $(\tilde {P})$ by replacing the full
actuation matrix \(H\) with \(H_{\mathcal N\mathcal I}\) in the voltage
certificate constraints, and with \(H_{\mathcal I\mathcal I}\) in the
inverter feedback constraints. Buses without inverters are therefore still certified, even
though they are not directly actuated. The main limitation relative to the
fully actuated case is that \(H_{\mathcal N\mathcal I}\) is generally
rectangular, and hence the exact disturbance cancellation argument based on
\(H^{-1}\) is no longer directly available.

\section{Numerical Studies}

\subsection{5-bus European residential feeder with full inverter placement}
\begin{figure}[h]
    \centering
    \includegraphics[width=0.45\textwidth]{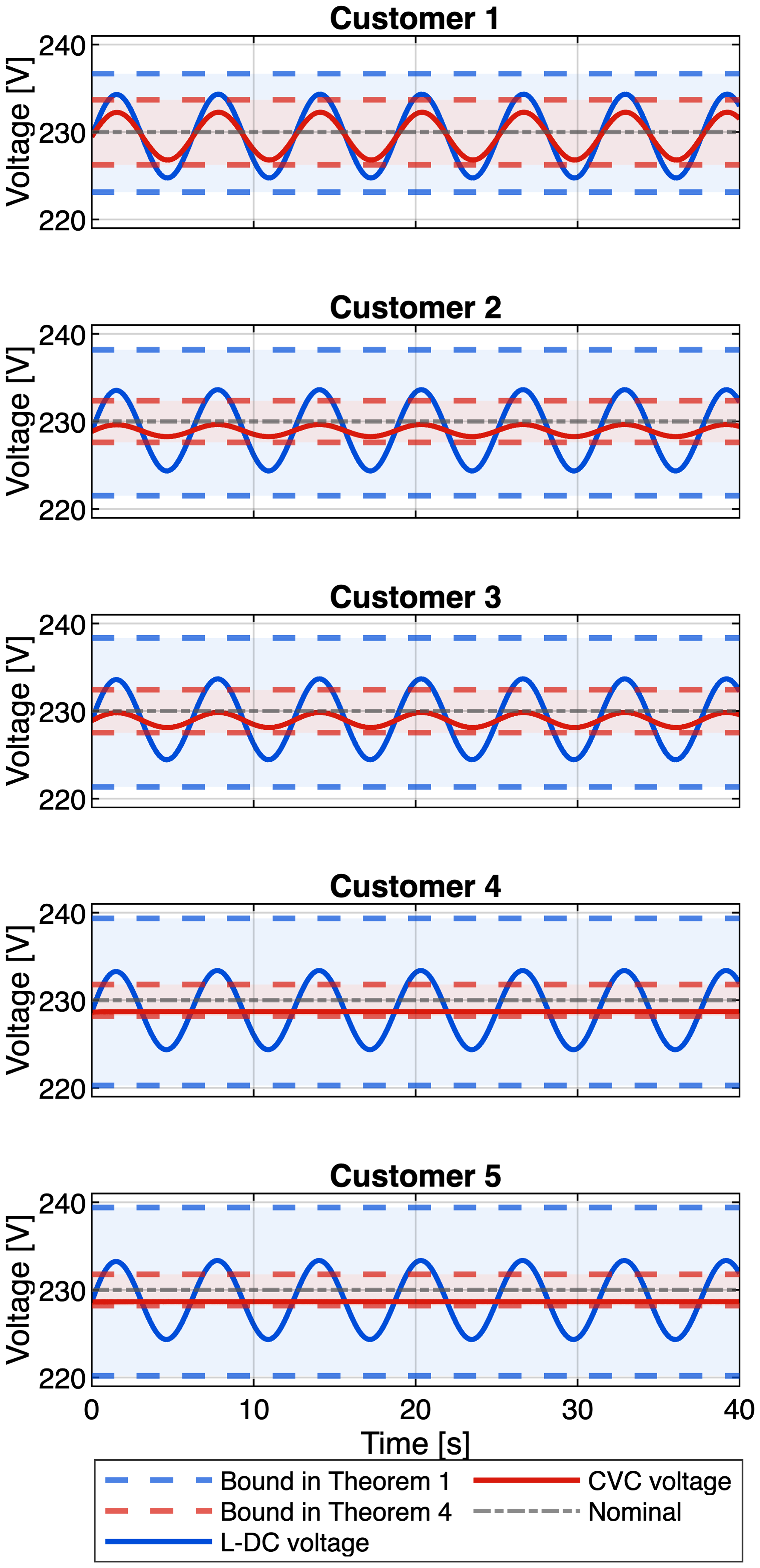}
    \caption{Individual customer voltages and deterministic certificates
  for the five customer feeder. Blue dashed lines denote the heterogeneous
  certificate of Theorem~1, and red dashed lines denote the certificate
  of Theorem~4. Blue solid curves are generated by standard Local Droop Control (L-DC)
  and red solid curves by the Coordinated Virtual Controller (CVC) of Section~5.}
    \label{fig:five_customer_voltages}
\end{figure}

We first consider a benchmark 5-customer residential European low voltage feeder by the CIGRE taskforce \cite{force2014benchmark}. The square deviation of the substation voltage $\hat{v}_0$  from the nominal $\bar{v}$ is given by $y_0(t)=230^2-(230+5\sin(t))^2$. %As such, $y_0$ takes values in the interval $[-2325, 2275]$, i.e., $|y_0(t)| \leq 2325$.  
The customer generation, consumption and inverter parameters are
summarised in Table~\ref{tab:five_customer_power_data}. The available reactive power reserve is obtained by $Q^{\mathrm{max}}_i=\sqrt{\bar s_i^2-\rho_{g,i}^2}$. The line and customer connection impedances are reported in
Table~\ref{tab:five_customer_parameters}.

 The inverter time constants
are $\tau_i=1\,\mathrm{s}$ and all controller states are initialised at
zero. In accordance with the numerical study in \cite{chong2019local}, the droop design uses the maximum slope
$d=0.2325$ and the squared voltage breakpoint
$y_{\max}=14899.4\,\mathrm{V}^2$. In this section, we compare the homogeneous certificates (Corollary 1 and Theorem 1 in \cite{chong2019local}), the heterogeneous certificates of Theorem \ref{thm:heterogeneous-certificate-agent} and the certificates corresponding to the coordinated virtual controller (CVC). For the latter, we set $\alpha=0.25$.    Figure~\ref{fig:five_customer_voltages} compares the three voltage
certificates with the corresponding closed-loop voltages of the two
controllers. Each subfigure corresponds to one of the five customers. The blue dashed lines denote the
heterogeneous certificate of Theorem~\ref{thm:heterogeneous-certificate-agent}, which apply to
the local droop controller (L-DC). %Consequently,  the heterogeneous certificates reduce
%conservatism by retaining the bus dependent disturbance
%bounds, inverter parameters and voltage sensitivities.
For this case study, the homogeneous certificate of Corollary 1 and Theorem 1 in \cite{chong2019local} corresponds to the values 207V and 253V. The blue solid line in each subplot denotes the voltage under the L-DC. The red dashed line denotes the
certificate obtained from Theorem \ref{thm:neighborhood_sparse_synthesis}, and the
red solid curve is the voltage under the corresponding
coordinated virtual controller (CVC). Hence, the comparison between the blue
and red results reflects not only a change of certificate but also a change of controller. %All five customer voltages are displayed individually to show at which
%customers the coordinated virtual controller provides the largest
%improvement.

Note that the coordinated
virtual controller (CVC) substantially improves the simulated voltage control for
all customers. The maximum deviations per customer from the nominal
voltage under local droop control are
$[5.259,5.648,5.556, 5.640 ,5.660]^{\top}\mathrm{V}$,
whereas, under the coordinated controller, they are reduced to
$[3.207,1.733,1.878,1.354,1.386]^{\top}\mathrm{V}$.
Note that at Customers~4 and~5, the peak-to-peak voltage variations
decrease from $9.028\,\mathrm{V}$ and $9.012\,\mathrm{V}$ under local
droop to only $0.061\,\mathrm{V}$ and $0.045\,\mathrm{V}$ under the
CVC. Hence, the coordinated controller almost completely rejects
the periodic disturbance at the downstream customers, although a small
steady voltage offset remains. This, however, can only be achieved as long as the reactive power capacity of the network allows it and depends on the tuning of the linear program $(\mathrm{\tilde{P}})$. \Autom{How the tuning parameters can be further optimized for better and more fair collective performance is left for future research.} Finally, we note that all L-DC trajectories remain within the heterogeneous
certificates of Theorem~1, while all CVC trajectories remain
within the corresponding certificate of Theorem~4.

\begin{table}[t]
\centering
\caption{Power for the 5-customer feeder in \cite{chong2019local}}
\label{tab:five_customer_power_data}
\begin{tabular}{c|ccccc}
\hline
Customer $i$ & 1 & 2 & 3 & 4 & 5 \\
\hline
$\rho_{g,i}$ [W]
& 3500 & 5500 & 4000 & 4500 & 3000 \\

$\rho_{c,i}$ [W]
& 2295 & 5440 & 5440 & 2295 & 2720 \\

$q_{c,i}$ [VAr]
& 300 & 960 & 480 & 600 & 400 \\

$\bar s_i$ [VA]
& 4200 & 6500 & 4700 & 5300 & 3600 \\

$Q^{\mathrm{max}}_i$ [VAr]
& 2321.6 & 3464.1 & 2467.8 & 2800.0 & 1990.0 \\
\hline
\end{tabular}
\end{table}

\begin{table}[t]
\centering
\caption{Line impedances for the 5-customer feeder in \cite{chong2019local}}
\label{tab:five_customer_parameters}
\begin{tabular}{c c c c c}
\hline
Customer $i$ & $R_i$ & $X_i$ & $R_i'$ & $X_i'$ \\
\hline
1  & 0.00343 & 0.04711 & 0.00147 & 0.02157 \\
2  & 0.00172 & 0.02356 & 0.00662 & 0.09707 \\
3  & 0.00343 & 0.04711 & 0.00147 & 0.02157 \\
4  & 0.00515 & 0.07067 & 0.00147 & 0.02157 \\
5  & 0.00172 & 0.02356 & 0.00147 & 0.02157 \\
\hline
\end{tabular}
\end{table}

\subsection{26-bus rural feeder with sparse inverter placement}
\label{subsec:simbench_case}

The second study considers an extended rural LV network,
constructed based on the SimBench feeder
\texttt{1-LV-rural1-0-no\_sw} provided in \cite{meinecke2020simbench}. Specifically, the SimBench data are used to obtain
the line parameters, customer data, photovoltaic placement
and time-series profiles. To create a more demanding radial grid test,
we double the depth of each of the four outgoing feeders by
appending a copy of its original sequence downstream of its final customer. The four radial paths of the
extended network are shown in Figure 5.  The customer set is $\mathcal N=
 \{1,15,\,
 8,11,10,3,16,17,18,19,
 7,$$12,14,6,5,20$,$21,22,$ $23,24,
 2,9,13,25,26,27\}$. Bus~4 is the common feeder head which delivers power to the four radial feeders with $26$ customer
buses in total. Its voltage \(v_0(t)\) is applied at the head of each of the four feeders.
%Consequently, the transformer voltage drop and any dependence of
%\(v_0(t)\) on the aggregate feeder power are not included in the
%sensitivity model. 
After ordering the buses per feeder, the network
sensitivity matrix can then be written in the form
$H  =
    \operatorname{blkdiag}
    \left(
        H_{F_1}, H_{F_2}, H_{F_3}, H_{F_4}
    \right)$,
where each \(H_{F_k}\) is the single line feeder sensitivity as derived
in Section~2.

The extended model initially contains eight controllable
inverters with index set $\mathcal I'
 = \{8,11,16,17,7,20,13,27\}$,
highlighted in pink in Figure 5. Note that their placement is highly nonuniform. Specifically, feeder $\mathrm{F1}$
contains no controllable inverter, while the final nodes of the two
longest feeders remain weakly actuated.  This topology illustrates the sparse inverter placement setting addressed in
Section~5.3. The proposed method certifies all $26$ customer voltages
although only a subset of buses has droop control. In Figures~6 and~7, each controllable inverter can communicate with
all other controllable inverters on the same feeder, while communication
between feeders is excluded.
\begin{figure}[t]
\centering
\begin{tikzpicture}[
    x=0.78cm,
    y=0.8cm,
    bus/.style={
        circle,
        draw,
        thick,
        minimum size=6.2mm,
        inner sep=0pt,
        font=\small
    },
    rootbus/.style={
        circle,
        draw,
        thick,
        fill=gray!20,
        minimum size=6.8mm,
        inner sep=0pt,
        font=\small\bfseries
    },
    pvbus/.style={
        circle,
        draw,
        thick,
        fill=purple!18,
        minimum size=6.2mm,
        inner sep=0pt,
        font=\small\bfseries
    },
    legendroot/.style={
        circle,
        draw,
        thick,
        fill=gray!20,
        minimum size=4.5mm,
        inner sep=0pt
    },
    legendpv/.style={
        circle,
        draw,
        thick,
        fill=purple!18,
        minimum size=4.5mm,
        inner sep=0pt
    },
    legendbus/.style={
        circle,
        draw,
        thick,
        minimum size=4.5mm,
        inner sep=0pt
    },
    line/.style={thick},
    lab/.style={font=\small},
    legendtext/.style={
        font=\scriptsize,
        inner sep=0pt,
        anchor=west
    }
]

% Transformer and root
\node[
    draw,
    thick,
    rectangle,
    rounded corners=1mm,
    minimum width=1.0cm,
    minimum height=0.45cm,
    font=\small
] (tr) at (0,1.0) {Transformer};

\node[rootbus] (b4) at (0,-0.2) {4};
\node[lab, above right=1mm and 1mm of b4] {Root Bus};

\draw[line] (tr) -- (b4);

% First layer from Bus 4
\node[bus]   (b1) at (-3.2,-1.1) {1};
\node[pvbus] (b8) at (-1.05,-1.1) {8};
\node[pvbus] (b7) at (1.05,-1.1) {7};
\node[bus]   (b2) at (3.2,-1.1) {2};

\draw[line] (b4) -- (b1);
\draw[line] (b4) -- (b8);
\draw[line] (b4) -- (b7);
\draw[line] (b4) -- (b2);

% Feeder 1: 4 -> 1 -> 15
\node[bus] (b15) at (-3.2,-2.2) {15};

\draw[line] (b1) -- (b15);

% Feeder 2: 4 -> 8 -> 11 -> 10 -> 3 -> 16 -> 17 -> 18 -> 19
\node[pvbus] (b11) at (-1.05,-2.2) {11};
\node[bus]   (b10) at (-1.05,-3.3) {10};
\node[bus]   (b3)  at (-1.05,-4.4) {3};
\node[pvbus] (b16) at (-1.05,-5.5) {16};
\node[pvbus] (b17) at (-1.05,-6.6) {17};
\node[bus]   (b18) at (-1.05,-7.7) {18};
\node[bus]   (b19) at (-1.05,-8.8) {19};

\draw[line] (b8)  -- (b11);
\draw[line] (b11) -- (b10);
\draw[line] (b10) -- (b3);
\draw[line] (b3)  -- (b16);
\draw[line] (b16) -- (b17);
\draw[line] (b17) -- (b18);
\draw[line] (b18) -- (b19);

% Feeder 3: 4 -> 7 -> 12 -> 14 -> 6 -> 5
%           -> 20 -> 21 -> 22 -> 23 -> 24
\node[bus]   (b12) at (1.05,-2.2) {12};
\node[bus]   (b14) at (1.05,-3.3) {14};
\node[bus]   (b6)  at (1.05,-4.4) {6};
\node[bus]   (b5)  at (1.05,-5.5) {5};
\node[pvbus] (b20) at (1.05,-6.6) {20};
\node[bus]   (b21) at (1.05,-7.7) {21};
\node[bus]   (b22) at (1.05,-8.8) {22};
\node[bus]   (b23) at (1.05,-9.9) {23};
\node[bus]   (b24) at (1.05,-11.0) {24};

\draw[line] (b7)  -- (b12);
\draw[line] (b12) -- (b14);
\draw[line] (b14) -- (b6);
\draw[line] (b6)  -- (b5);
\draw[line] (b5)  -- (b20);
\draw[line] (b20) -- (b21);
\draw[line] (b21) -- (b22);
\draw[line] (b22) -- (b23);
\draw[line] (b23) -- (b24);

% Feeder 4: 4 -> 2 -> 9 -> 13 -> 25 -> 26 -> 27
\node[bus]   (b9)  at (3.2,-2.2) {9};
\node[pvbus] (b13) at (3.2,-3.3) {13};
\node[bus]   (b25) at (3.2,-4.4) {25};
\node[bus]   (b26) at (3.2,-5.5) {26};
\node[pvbus] (b27) at (3.2,-6.6) {27};

\draw[line] (b2)  -- (b9);
\draw[line] (b9)  -- (b13);
\draw[line] (b13) -- (b25);
\draw[line] (b25) -- (b26);
\draw[line] (b26) -- (b27);

% Feeder labels
\node[lab] at (-3.2,-2.9) {F1};
\node[lab] at (-0.45,-9.4) {F2};
\node[lab] at (1.65,-11.5) {F3};
\node[lab] at (3.75,-7.2) {F4};

% Compact single-line legend
\node[legendroot] (legroot) at (-2.75,-12.6) {};
\node[legendtext, right=0.4mm of legroot]
    (legroottext) {root};

\node[legendpv, right=1.5mm of legroottext]
    (legpv) {};
\node[legendtext, right=0.4mm of legpv]
    (legpvtext) {PV bus};

\node[legendbus, right=1.5mm of legpvtext]
    (legbus) {};
\node[legendtext, right=0.4mm of legbus]
    {customer bus};

\end{tikzpicture}

\caption{Network based on the SimBench LV rural feeder topology. The controllable PV inverter buses are highlighted in light pink.}
\label{fig:simbench_topology}
\end{figure}
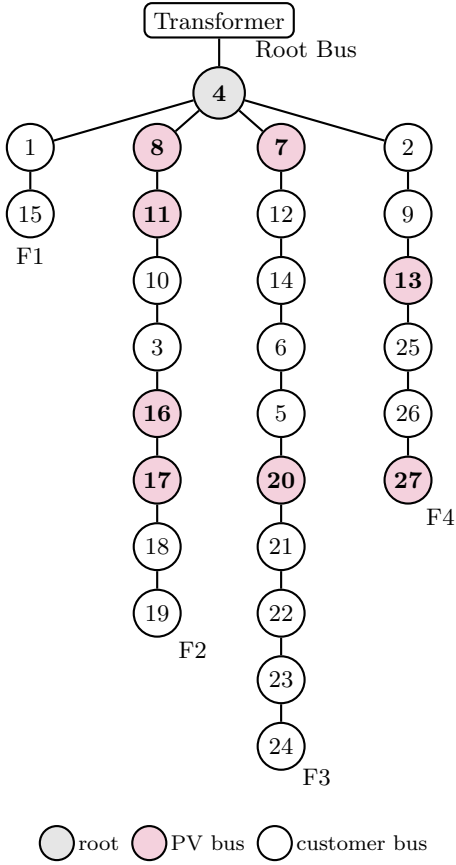

The SimBench load and PV trajectories are used without modifying their ordering. The componentwise disturbance bounds are obtained from \texttt{LoadProfile.csv} and
\texttt{RESProfile.csv}.  Since the data set does not provide smart inverter Volt/VAR
parameters, we assume for each
original PV system that $S_i=\max\{s_{R,i},1.25P_i^{\max}\}$, giving a reactive power reserve of at most 
 $Q_i^{\max}=\sqrt{S_i^2-(P_i^{\max})^2}$.
The droop
slope is then selected as $d_i=
 0.70\,
 \dfrac{Q_i^{\max}}
 {\bar v^2-(0.95\bar v)^2}$, while 
the inverter time constants are $\tau_i=10\,\mathrm{s}$, and $v_0(t)=
 \bar v+\dfrac{5}{\sqrt{3}}
 \sin\left(\dfrac{2\pi t}{3600}\right)\mathrm{V}.$

\begin{figure}[h]
    \centering
    \includegraphics[width=0.49\textwidth]{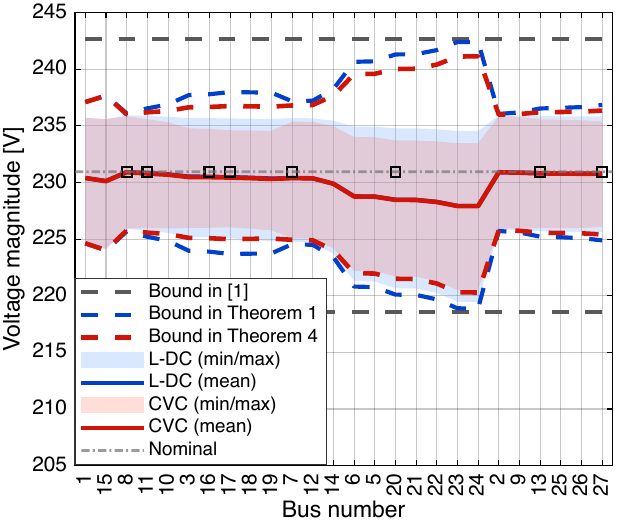}
    \caption{Voltage certificates and simulated voltages for the
initial sparse inverter set \(\mathcal I'\). Buses are ordered
per feeder according to Figure ~5.}
    \label{fig:without_addition}
\end{figure}
\begin{figure}[h]
    \centering
    \includegraphics[width=0.49\textwidth]{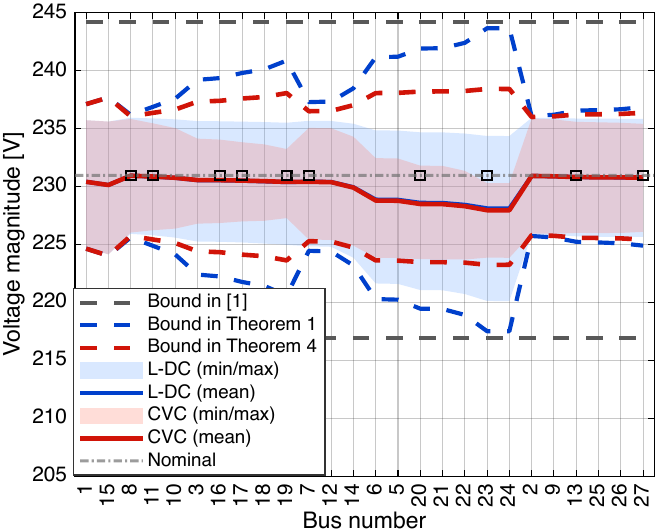}
    \caption{Voltage certificates and simulated voltages after adding
controllable inverters at Buses~19 and~23. Note that, even though the topology for inverters is still sparse, our method exhibits significant improvements downstream and more uniform certificates compared to local droop control. }
    \label{fig:with_addition}
\end{figure}

 An advantage of the heterogeneous certificates in Theorem \ref{thm:heterogeneous-certificate-agent} is that they reveal where the existing inverter placement provides
insufficient voltage control. This can be used for further optimization of the inverter placement problem for better  designs. %Developing a detailed inverter placement methodology is outside the
%scope of the present contribution and is left for future work. 
Motivated by this information, two
additional inverters are introduced at Buses~19 and~23, which gives rise to the inverter index set $\mathcal I
 =
 \mathcal I'\cup\{19,23\}.$
These additions provide support near the ends of the two longest feeders, while the network still remains sparsely actuated.
Figure~\ref{fig:without_addition} shows the voltage certificates before adding the
two downstream inverters 19 and 23. The homogeneous certificate assigns the same
conservative envelope to every bus, whereas the heterogeneous
certificate of Theorem~1 reveals that the bounds widen towards the ends of the long feeders. The
certificate in Theorem \ref{thm:neighborhood_sparse_synthesis} is tighter, but the simulated voltage ranges of
local droop (L-DC) and the coordinated virtual controller  (CVC) remain relatively
close with only small improvements in both the simulated voltages and the theoretical certificate. This indicates that coordinated controller synthesis alone cannot fully
compensate for the limited voltage authority because of
sparse inverter placement, given the reactive power capacity of this network. 

Figure~\ref{fig:with_addition} shows the same comparison after
adding controllable inverters at Buses~19 and~23 which again follow the paradigm of actuators obtained by SimBench \cite{meinecke2020simbench}. Again, the heterogeneous certificate does not automatically improve and becomes
particularly wide along the long feeders, since the additional local
droop devices also contribute to the bound. In contrast,
the certificates of Theorem~\ref{thm:neighborhood_sparse_synthesis} remain substantially tighter and more uniform downstream.
Moreover, the CVC minimum and maximum voltage range is visibly narrower than the
L-DC range at the downstream buses. The two figures show that addition of inverters often provides the
necessary downstream voltage capacity, but its benefit is realised
only when the inverter actions are designed jointly via the LP optimization problem $(\mathrm{\tilde P})$. Thus, our proposed
method uses the additional degrees of freedom to reshape the
network sensitivity, thus reducing both the voltage certificates and
the voltage deviations, and certifying all monitored buses while
respecting the operational constraints of the system.

\section{Conclusion}

We improved the worst-case certificates in \cite{chong2019local}, \cite{Chong2020} deriving tighter heterogeneous voltage certificates for
droop control of radial networks. We then leveraged them to efficiently synthesize
coordinated droop controllers that enjoy tighter guarantees via a collective linear program. Our linear program accounts
for inverter limits, possible communication limits, prosumer disturbances, and possible sparse actuation.
In the 5-customer benchmark study, the coordinated droop  controllers reduced significantly the
voltage deviations as well as the associated certificates compared with local
droop control. In the 26-bus study, we illustrated  that additional inverters are more effective when their actions are
designed jointly through the proposed linear program.  The development of
fairness mechanisms to prevent improvements at some buses
from being obtained at the expense of others is left for future research.  Future work will also focus on the development of a data-driven methodology  based on adversarially robust statistical learning methods for reachability analysis \cite{Pantazis2026,Campi2025}.

\section{Appendix}

\emph{Proof of Lemma \ref{lem:exogenous-bound-customer-i}}: 
Using the expression for $\phi_i(\rho,q_c)$ derived in
Section~2.2 from~(\ref{eq:Pi}), the recursion~(\ref{eq:recursion}),
and the nonnegativity of the line parameters, the triangle
inequality gives
\begin{align}
|\phi_i(\rho,q_c)|
&\leq
2\sum_{k=0}^{i-1}
\left(
R_k\sum_{j=k+1}^{N}|\rho_j|
+
X_k\sum_{j=k+1}^{N}|q_{c,j}|
\right) \nonumber\\
&\quad+
2R'_{i-1}|\rho_i|
+
2X'_{i-1}|q_{c,i}| \nonumber\\
&\leq
2\sum_{k=0}^{i-1}
\left(
R_k\sum_{j=k+1}^{N}\bar\rho_j
+
X_k\sum_{j=k+1}^{N}\bar q_{c,j}
\right) \nonumber\\
&\quad+
2R'_{i-1}\bar\rho_i
+
2X'_{i-1}\bar q_{c,i}
=
\bar\phi_i,
\end{align}
where the second inequality follows from
Assumption~\ref{assum:exogenous-bounds}.
$\blacksquare$

\emph{Proof of Theorem \ref{thm:heterogeneous-certificate-agent}}: For simplicity, we denote $e_j(t):=\exp\!\left(-\frac{t}{\tau_j}\right)$. For each agent $j\in\mathcal N$, the time response of the system according to the inverter dynamics in  (\ref{state_dynamics}) is then given by:
\[
q_{g,j}(t)
=
e_j(t)q_{g,j}(0)
+
\frac{1}{\tau_j}
\int_0^t
e_j(t-s)
K_j(y_j(s))\,ds .
\]
From Assumption \ref{assum:sector-bounded}, it holds that:
\begin{align}
|q_{g,j}(t)|
&\le
e_j(t)|q_{g,j}(0)|
+
\frac{1}{\tau_j}
\int_0^t
e_j(t-s)
\bar K_j\,ds \nonumber \\
&=
e_j(t)|q_{g,j}(0)|
+
\bigl(1-e_j(t)\bigr)\bar K_j \nonumber \\
&\le
\max\{|q_{g,j}(0)|,\bar K_j\}
=
\bar Q_j .
\label{eq:q_agent_uniform_bound}
\end{align}
Applying \eqref{eq:q_agent_uniform_bound} to (\ref{output_dynamics}), it holds that:
\begin{align}
|y_j(s)|
&=
\left|
\sum_{k=1}^N h_{jk}q_{g,k}(s)
+
\phi_j(\rho,q_c)
+
y_0
\right| \nonumber \\
&\le
\sum_{k=1}^N \widetilde h_{jk}|q_{g,k}(s)|
+
\ell_j . \nonumber 
\end{align}
Using \eqref{eq:q_agent_uniform_bound} for all $k\in\mathcal N$ we obtain the inequality:
\begin{align}
|y_j(s)|
\le
\sum_{k=1}^N \widetilde h_{jk}\bar Q_k+\ell_j
=
\gamma_j . \nonumber 
%\label{eq:y_agent_gamma_bound}
\end{align}
Moreover, by Assumption \ref{assum:sector-bounded}, it holds that $|K_j(y_j(s))|
=
|K_j(y_j(s))-K_j(0)|
\le
d_j|y_j(s)|
\le
d_j\gamma_j$ . Since $|K_j(y_j)| \leq  \bar K_j $ for all $y_j \in \mathbb{R}$,
we have that $|K_j(y_j)| \leq m_j$, where $m_j=\min\{\bar{K}_j, d_j\gamma_j\}$.
From the time-response of \eqref{state_dynamics} we obtain the bound:
\begin{align}
|q_{g,j}(t)|
&\le
e_j(t)|q_{g,j}(0)|
+
\frac{m_j}{\tau_j}
\int_0^t
e_j(t-s) \,ds \nonumber \\
&=
e_j(t)|q_{g,j}(0)|
+
\bigl(1-e_j(t)\bigr)m_j .
\label{eq:q_agent_bound}
\end{align}
 Since
$y_i(t)
=
\sum\limits_{j=1}^N h_{ij}q_{g,j}(t)
+
\phi_i(\rho,q_c)
+
y_0$, by analyzing \eqref{output_dynamics} per component,
we have $
|y_i(t)|
\le
\sum\limits_{j=1}^N \widetilde h_{ij}|q_{g,j}(t)|
+
\ell_i$ and using \eqref{eq:q_agent_bound} for each $j\in\mathcal N$ we obtain the inequality:
\[
|y_i(t)|
\le
\sum_{j=1}^N
\widetilde h_{ij}
\left[
e_j(t)|q_{g,j}(0)|
+
\bigl(1-e_j(t)\bigr) m_j
\right]
+
\ell_i .
\]
Finally, since $e_j(t)\in[0,1]$, considering the extended real line $t \in \bar{\mathbb{R}}_{+}=[0, \infty]$, it holds that $e_j(t)|q_{g,j}(0)|
+
\bigl(1-e_j(t)\bigr)m_j
\le
\max\!\left\{
|q_{g,j}(0)|,\,m_j
\right\}.$
Therefore,
\[
|y_i(t)|
\le
\sum_{j=1}^N
\widetilde h_{ij}
\max\!\left\{
|q_{g,j}(0)|,\,
m_j
\right\}
+
\ell_i,
\]
thus concluding the proof. \hfill $\blacksquare$ \\
\emph{Proof Sketch for Corollary \ref{cor:homogeneous-special-case-customer-i}}: 
From Theorem~1, using
$d_j\leq d$, $\ell_j\leq\bar{\ell}$,
$\bar Q_j\leq\|q_g(0)\|_\infty+\bar K$, and
$\sum_{j=1}^N \widetilde h_{ij}\leq\|\widetilde H\|_\infty$,
we obtain
\begin{align}
|y_i(t)|
&\leq \|\widetilde H\|_\infty\|q_g(0)\|_\infty
+\frac{d\tau}{\tau_m}\|\widetilde H\|_\infty\bar{\ell} \nonumber\\
&\quad
+\frac{d\tau}{\tau_m}\|\widetilde H\|_\infty^2
\left(\|q_g(0)\|_\infty+\bar K\right)
+\bar{\ell},
\end{align}
thus concluding the proof. \hfill 
$\blacksquare$

\emph{Proof of Lemma 2:} Let $B\in\mathbb{R}^{N\times N}$ be the unit lower-triangular matrix, i.e., $B_{ij}=1$ if $j\leq i$ and $B_{ij}=0$ otherwise, and define $\mathcal{X}:=\operatorname{diag}(X_0,\ldots,X_{N-1})$. Then, for every $i,j\in\mathcal{N}$,
\begin{align}
(B\mathcal{X}B^\top)_{ij}
&=\sum_{k=1}^{\min\{i,j\}} X_{k-1}
=\sum_{\ell=0}^{\min\{i,j\}-1} X_\ell
=S_{ij}.
\end{align}
Hence, $S=B\mathcal{X}B^\top$. Since $B$ is nonsingular and $X_i>0$ for all $i\in\mathcal{N}_{-}$, \Autom{it holds that} $\mathcal{X}\succ0$ and thus $S\succ0$. Moreover, since $X'_i\geq0$, \Autom{we have that}
$D':=\operatorname{diag}(X'_0,\ldots,X'_{N-1})\succeq0,$
\Autom{which implies that} $H=-2S-2D'\prec0$. \hfill $\blacksquare$

\emph{Proof of Lemma \ref{lem:invariance}}: 
Let $A=D|F|$. Since $A\geq0$ and $\rho(A)<1$ by
Assumption~5, $(I-A)^{-1}=\sum_{k=0}^{\infty}A^k\geq0$.
Hence, with $r_z=(I-A)^{-1}\bigl(|z(0)|+D\tilde{\ell}\bigr)$,
we obtain
$r_z=Ar_z+|z(0)|+D\tilde{\ell}$.
Therefore,
$r_z\geq |z(0)|$ and
$r_z\geq D(|F|r_z+\tilde{\ell})$,
which proves the first claim. For the second claim, let $z\in\mathcal{Z}$. From the output
dynamics and Assumption~4,
\begin{align}
|y_i|
&\leq
\sum_{j=1}^{N}|F_{ij}|r_{z,j}+\tilde{\ell}_i.
\end{align}
By Assumption~1, it holds that: $|K_i(y_i)|
\leq d_i|y_i|
\leq r_{z,i}$.
Thus, on $z_i=r_{z,i}$, $\dot z_i\leq0$, while on
$z_i=-r_{z,i}$, $\dot z_i\geq0$. Since
$r_z\geq|z(0)|$, the initial condition belongs to
$\mathcal{Z}$ and the vector field points inward on every
boundary face. Hence $\mathcal{Z}$ is forward invariant.
$\blacksquare$

\bibliographystyle{unsrt}
\bibliography{autosam}    % and a bib file to produce the 
                                 % bibliography (preferred). The
                                 % correct style is generated by
                                 % Elsevier at the time of printing.

\end{document}